\documentclass[11pt]{article}
\usepackage[utf8]{inputenc}
\usepackage[T1]{fontenc}
\usepackage{graphicx}
\usepackage{amsmath,amsfonts,amssymb,amsthm}
\usepackage{natbib}

\usepackage{mathtools}
\usepackage{mathrsfs}
\usepackage{xcolor}
\usepackage{hyperref}
\usepackage{booktabs}
\usepackage{bm}
\usepackage{color}
\usepackage[font=footnotesize]{caption}
\usepackage{enumitem}
\usepackage{empheq}
\usepackage{framed}
\usepackage{comment}
\usepackage{fullpage}
\usepackage[capitalize]{cleveref}
\usepackage{soul}
\usepackage{dirtytalk}
\usepackage{dsfont}
\usepackage{bbm}
\usepackage{thmtools}
\usepackage{thm-restate}

\definecolor{myred}{HTML}{880000}
\definecolor{mygreen}{HTML}{008800}
\definecolor{myblue}{HTML}{000088}
\definecolor{linkblue}{HTML}{0000BB}

\hypersetup{
  colorlinks,
  linkcolor={myred},
  citecolor={myblue},
  urlcolor={linkblue}
}

\newcommand{\N}{\mathbb N}

\newcommand{\R}{\mathbb R}

\newcommand{\E}{{\mathbf E}}
\renewcommand{\P}{\mathbf P}

\newcommand{\pp}{\,:\,}
\newcommand{\vol}{\mathrm{Vol}}

\renewcommand{\leq}{\leqslant}
\renewcommand{\geq}{\geqslant}
\renewcommand{\le}{\leqslant}
\renewcommand{\ge}{\geqslant}

\newcommand{\argmax}{\mathop{\mathrm{arg}\,\mathrm{max}}}
\newcommand{\wt}{\widetilde}
\newcommand{\wh}{\widehat}
\newcommand{\ol}{\overline}

\newcommand{\di}{\mathrm{d}}

\newcommand{\ind}[1]{\bm 1 ( #1 )}

\newcommand{\eps}{\varepsilon}

\newcommand{\probas}{\mathcal{P}}

\newcommand{\M}{M}
\newcommand{\X}{\mathcal{X}}

\newcommand{\Qs}{\mathcal{Q}}

\newcommand{\nloc}{N_{\mathsf{loc}}}
\newcommand{\conv}{\mathop{\mathrm{conv}}}

\usepackage{xspace}
\newcommand{\ie}{i.e.\@\xspace}
\newcommand{\eg}{e.g.\@\xspace}
\newcommand{\iid}{i.i.d.\@\xspace}

\renewcommand{\triangleq}{=}
\renewcommand{\hat}{\widehat}
\renewcommand{\tilde}{\widetilde}

\newcommand{\kl}{\operatorname{KL}}
\newcommand{\kll}[2]{\kl ({#1}, {#2})}

\newcommand{\hel}{\operatorname{H}}
\newcommand{\hels}{\hel^2}

\newtheorem{proposition}{Proposition}
\newtheorem{theorem}{Theorem}
\newtheorem*{theorem*}{Theorem}
\newtheorem{lemma}{Lemma}
\newtheorem{claim}{Claim}
\newtheorem{corollary}{Corollary}
\newtheorem{fact}{Fact}

\theoremstyle{definition}
\newtheorem{definition}{Definition}

\theoremstyle{remark}

\DeclarePairedDelimiter{\abs}{\lvert}{\rvert}

\DeclarePairedDelimiter{\set}{\{}{\}}

\title{Ratio Covers of Convex Sets and Optimal Mixture Density Estimation}
\author{
Spencer Compton\thanks{Department of Computer Science, Stanford University.}\quad
G\'{a}bor Lugosi\thanks{Department of Economics and Business, Pompeu Fabra University, Barcelona, Spain; ICREA, Barcelona, Spain; Barcelona Graduate School of Economics.}\quad
Jaouad Mourtada\thanks{Department of Statistics, CREST/ENSAE Paris, Palaiseau, France.}\quad
Jian Qian\thanks{Department of AI and Data Science, The University of Hong Kong.}\\
Nikita Zhivotovskiy\thanks{Department of Statistics, University of California, Berkeley.}
}
\date{}

\begin{document}

\maketitle

\begin{abstract}
  We study density estimation in Kullback-Leibler divergence: given an i.i.d.~sample from an unknown density $p^\star$, the goal is to construct an estimator $\widehat{p}$ such that $\mathrm{KL} (p^\star, \widehat{p})$ is small with high probability.
  We consider two fundamental settings involving a finite dictionary of densities: (i) \emph{model aggregation}, where $p^\star$ belongs to the dictionary, and (ii) \emph{convex aggregation} (mixture density estimation), where $p^\star$ is a mixture of densities from the dictionary.
  Crucially, we make no assumption on the base densities: their ratios may be unbounded and their supports may differ.
  For both problems, we identify the best possible high-probability guarantees in terms of the dictionary size, sample size, and confidence level.
  These optimal rates are higher than those achievable when density ratios are bounded by absolute constants; for mixture density estimation, they match existing lower bounds in the special case of discrete distributions.
  
  Our analysis of the mixture case hinges on two new covering results.
  First, we provide a sharp, distribution-free upper bound on the local Hellinger entropy of the class of mixtures of $\M$ distributions.
  Second, we prove an optimal \emph{ratio covering} theorem for convex sets: for every convex compact set $K \subset \mathbb{R}_+^d$, there exists a subset $A \subset K$ with at most $2^{O(d)}$ elements such that each element of $K$ is coordinate-wise dominated by an element of $A$ up to a universal constant factor.
  This geometric result is of independent interest; notably, it yields new cardinality estimates for $\varepsilon$-approximate Pareto sets in multi-objective optimization with convex feasible set.
\end{abstract}

\setcounter{tocdepth}{1}
\section{Introduction}

We revisit two classical distribution estimation problems.
Let $\X$ be a measurable space endowed with a measure $\mu$.
In addition, let $p_1,\ldots,p_\M$ be a collection of $\M \geq 2$ known densities on $\X$ with respect to $\mu$, with no assumptions on these densities (in particular, the corresponding distributions need not satisfy any moment or tail conditions, and their supports may differ).
We 
observe an \iid sample $X_1,\ldots,X_n$ drawn from some unknown density $p^\star$ for which we only assume either
\begin{enumerate}
\item $p^\star \in \{p_1,\ldots,p_\M\}$, in the \emph{Model Aggregation} setting;
\item or $p^\star \in \operatorname{conv}\{p_1,\ldots,p_\M\}$, in the \emph{Convex Aggregation} (or \emph{mixture density estimation}) setting.
  % (Convex Aggregation, also known as mixture density estimation).
\end{enumerate}
Our goal is to construct, based on $X_1,\ldots,X_n$, a density estimator $\widehat{p}$ for which the Kullback--Leibler divergence $\kl(p^\star,\widehat{p})$ is small with high probability (where we recall that $\kll{p}{q} = \int_\X p \log (p/q) \di \mu$ for two probability densities $p, q$ on $\X$ with respect to $\mu$).

In this work, our primary aim is to characterize the best possible high-probability guarantees for both Model Aggregation and Convex Aggregation.
More precisely, given $\M, n \geq 2$ and $\delta \in (0, 1/2)$, we aim to identify (up to absolute constant factors) the smallest quantities $\psi_{\mathrm{MA}} (M, n, \delta)$ and $\psi_{\mathrm{CA}} (M, n, \delta)$ such that, for every probability densities $p_1, \dots, p_\M$ on a measure space $(\X, \mu)$:
\begin{enumerate}
\item there exists an estimator $\wh p_{\mathrm{MA}}$ such that, for every $p^\star \in \set{p_1, \dots, p_\M}$, with probability at least $1-\delta$ over $X_1, \dots, X_n$ \iid from $p^\star$ one has
  \begin{equation}
    \label{eq:def-optrate-ma}
    \kll{p^\star}{\wh p_{\mathrm{MA}}}
    \leq \psi_{\mathrm{MA}} (M, n, \delta)
    ;
  \end{equation}
\item there exists an estimator $\wh p_{\mathrm{CA}}$ such that, for every $p^\star \in \conv \set{p_1, \dots, p_\M}$, with probability at least $1-\delta$ over $X_1, \dots, X_n$ \iid from $p^\star$ one has
  \begin{equation}
    \label{eq:def-optrate-ca}
    \kll{p^\star}{\wh p_{\mathrm{CA}}}
    \leq \psi_{\mathrm{CA}} (M, n, \delta)
    .
  \end{equation}
\end{enumerate}

The fact that we make no assumption on the densities\footnote{Given probability distributions $P_1, \dots, P_\M$ on $\X$, the existence of a dominating measure $\mu$ such that each $P_j$ admits a density $p_j$ with respect to $\mu$ is not restrictive, as one may simply take $\mu = P_1 + \dots + P_\M$.} $p_1, \dots, p_\M$ induces some difficulties even for Model Aggregation.
Indeed, in this case it is straightforward to show (see Fact~\ref{fac:model-selection-fails} in Appendix~\ref{sec:suboptimal}) that no \emph{selection method} $\wh p$ returning an element of the dictionary $\set{p_1, \dots, p_\M}$ achieves any meaningful distribution-free guarantee.
In particular, this rules out arguably the simplest method, namely the maximum likelihood estimator (MLE).
For this reason, the procedures we will consider will return mixtures of distributions from the dictionary, even for Model Aggregation.

Aside from the MLE, the most common method for Model Aggregation is Bayesian model averaging, which returns a mixture of the densities $p_1, \dots, p_\M$ weighted by their Bayesian posterior weights (say, starting from a uniform prior distribution over $p_1, \dots, p_\M$).
In Appendix~\ref{sec:suboptimal}, we show that this standard method is also suboptimal in the high-probability regime.
Specifically, we obtain matching upper and lower bound on the tail performance of Bayes model averaging in the case where $M=2$, which show that this method achieves nontrivial yet suboptimal guarantees.

We note that there is a rich literature on Kullback-Leibler density estimation, both in the Model aggregation~\cite{barron1987bayes,yang1999information,yang2000mixing,catoni1997mixture,juditsky2008learning,audibert2007progressive,audibert2009fastrates,Rigollet2012KLAggregation,ButuceaDelmasDutfoyFischer2017Optimal,baraud2021tests,vanderhoeven2023high} and Convex aggregation~\cite{li1999mixture,li1999estimation,rakhlin2005risk,dalalyan2018optimal} settings.
However, our point of view differs from that of the aggregation literature.
Indeed, a key emphasis in these works is the effect of model misspecification, where $p^\star$ does not belong to $\set{p_1, \dots, p_\M}$; while---with two exceptions discussed below---the ratios of densities $p_1, \dots, p_\M$ are assumed to be bounded, with the guarantees diverging when the bound on density ratios diverges.
In contrast, we consider arbitrary densities $p_1, \dots, p_\M$, for which optimal high-probability guarantees are unknown even in the well-specified case (where $p^\star$ satisfies either of the two assumptions above).
For these reasons, we focus in this work on optimal guarantees in the well-specified case, although we provide in Section~\ref{sec:misspecified} extensions to the misspecified case with additional logarithmic factors.

An important special case of mixture density estimation in which the boundedness condition fails to hold is that of \emph{estimation of discrete distributions}, where one aims to estimate an arbitrary distribution on a finite set (say, $\set{1, \dots, M}$).
This can be framed as mixture density estimation, where $\X = \set{1, \dots, M}$, $\mu$ is the counting measure (\ie, $\mu (A)$ is the cardinality of $A$ for $A \subset \X$), and for $1 \leq j \leq M$ the density $p_j : \X \to \R^+$ is defined by $p_j (k) = 1$ if $k = j$ and $p_j (k) = 0$ if $k \neq j$.
(In other words, any distribution on $\X$ is a mixture of Dirac masses.)
In this case, it is relatively straightforward to show, and has been known for a long time~\cite{catoni1997mixture,forster2002relative,braess2004bernstein,mourtada2019improper}, that optimal in-expectation guarantees are of order $\M/n$.
In contrast, the study of high-probability bounds is markedly more delicate and has received significant interest in recent years~\cite{bhattacharyya2021near,han2023optimal,canonne2023concentration,vanderhoeven2023high,mourtada2025estimation,van2025nearly}, leading finally to optimal guarantees in~\cite{mourtada2025estimation}.
However, the techniques used in these works are largely specific to the discrete setting, and
it is unclear how to extend them to the general mixture setting.

We now discuss existing guarantees in Model Aggregation and Convex Aggregation that apply to arbitrary densities $p_1, \dots, p_\M$---both of which only hold in expectation.

For Model Aggregation, optimal guarantees in expectation have been obtained in seminal work of Yang and Barron~\cite{yang1999information,yang2000mixing} and independently by Catoni~\cite{catoni1997mixture,catoni2004statistical}.
The corresponding estimator $\wt p$ (the progressive mixture rule) traces back to~\cite{barron1987bayes} and consists in a version of Bayes model averaging with additional averaging over sample sizes.
For $\M$ densities, it achieves the following guarantee:
\begin{equation}
    \label{eq:yangbarron}
    \E\big[\kl(p^\star,\widetilde{p})\big]\;\le\;\frac{\log \M}{n+1}.
\end{equation}
While this conclusively addresses the important question of optimal in-expectation guarantees, it leaves the more delicate question of optimal high-probability guarantees unanswered.
Specifically, the upper bound~\eqref{eq:yangbarron} implies that $\psi_{\mathrm{MA}} (M, n, \delta) \leq \log(M)/(n \delta)$, but this bound is highly suboptimal in the high-confidence regime of small $\delta$.
In fact, as we discuss in Appendix~\ref{sec:suboptimal}, the progressive mixture rule achieves suboptimal high-probability guarantees, hence another estimator $\wh p$ is required.

For Convex Aggregation, optimal guarantees are unknown even in expectation.
The best known general upper bound is achieved by a continuous version of the progressive mixture rule $\wt p$, applied to the probability simplex with uniform prior distribution.
Combining regret bounds for the Bayes mixture method~\cite{hazan2007logarithmic} (related log-loss regret guarantees for the simplex go back to Krichevsky--Trofimov \cite{krichevsky1981performance} and to the universal portfolio analysis of Cover \cite{cover1991universal}, see also~\cite{jezequel2025efficient} for recent progress on this problem) with an online-to-batch conversion gives the following guarantee:
\begin{equation}
    \label{eq:simplex}
\E\big[\kl(p^\star,\widetilde{p})\big]\le\frac{\M\big(1+\log(1+n/\M)\big)}{n+1}.
\end{equation}
This bound is known to be suboptimal at least in the canonical special case of discrete distribution estimation discussed above, due to the presence of the $\log (1+ n/M)$ factor.
In addition, as in the Model Aggregation case, it does not address the question of optimal high-probability guarantees.

\subsection{Our contributions}

Our main statistical contributions are optimal high-probability guarantees for both Model Aggregation and Convex Aggregation (as well as optimal in-expectation rates for Convex Aggregation), for arbitrary dictionary of densities.
In what follows $a \lesssim b$ means that there exists an absolute constant $c > 0$ such that $a \le cb$.

\begin{framed}
  \begin{enumerate}
  \item Theorem~\ref{thm:modelaggregation}: for any $M, n \geq 2$ and $\delta \in (e^{-n}, e^{-1})$, the optimal high-probability rate~\eqref{eq:def-optrate-ma} of Model Aggregation scales as
    \begin{equation*}
      \psi_{\mathrm{MA}} (M, n, \delta)
      \asymp \frac{\log M \log (1/\delta)}{n}.
    \end{equation*}
  \item Theorem~\ref{thm:convexaggregation}: for any $n \geq M \geq 2$ and $\delta \in (e^{-n}, e^{-1})$, the optimal high-probability rate~\eqref{eq:def-optrate-ca} of Convex Aggregation scales as
    \begin{equation*}
      \psi_{\mathrm{CA}} (M, n, \delta)
      \asymp \frac{M + \log M \log (1/\delta)}{n}.
    \end{equation*}
  \end{enumerate}
\end{framed}
Specifically, the lower bounds above have been established in the special case of discrete distributions~\cite{mourtada2025estimation,van2025nearly}; our contribution is to establish matching upper bounds for arbitrary dictionaries.

A standard approach to these problems is to use \textit{metric entropy} estimates: one derives statistical guarantees by analyzing coverings of the relevant class of distributions. As mentioned above, existing implementations of this strategy have not yielded sharp guarantees in our setting. In this work, we develop tools that make a metric-entropy approach sharp. There are two main obstacles: (i) we do not have sharp covering numbers for mixture models in KL divergence, and (ii) even with sharp covering numbers, the usual reductions would still not lead to sharp statistical bounds.

Our approach relies on three ingredients.
First, for both problems we construct a preliminary estimator that is close to $p^\star$ in Hellinger distance.
This metric is more amenable to entropy methods, but for the mixture case we need stronger local covering bounds than were previously available (to the best of our knowledge, this remained open; see \cite{baraud2018rho} for the state of the art). We establish the required entropy estimate (Theorem~\ref{thm:local-entropy-mixtures}) in Section~\ref{sec:local-cover-numb}.
Second, we construct an $\eps$-cover in KL divergence \emph{within} the resulting local Hellinger ball.
Third, we treat this KL cover as a finite set of candidates and feed it into the recent estimator of \cite{vanderhoeven2023high} under bounded density ratios, which returns a distribution whose KL risk is nearly as good as the best candidate in the set.
We briefly gloss over some additional technicalities. In particular, we mix candidates in a specific way to satisfy the conditions required by the estimator. For completeness, we provide a self-contained (less general) version of this estimator in \Cref{sec:highprobabilityonlinetobatch}.

The second ingredient, namely constructing a KL cover of a local Hellinger ball, is a critical step of the proof.
For reasons explained later, it reduces to the following more general geometric problem.
Given a convex compact subset $K \subset \R_+^d$, we would like to find a finite subset $A \subset K$ with controlled cardinality satisfying the following (one-sided) multiplicative approximation guarantee: for every $\theta = (\theta_j)_{1 \leq j \leq d} \in K$, there exists an element $\phi = (\phi_{j})_{1 \leq j \leq d} \in A$ such that $\theta_j \leq c \, \phi_j$ for every $j = 1, \dots, d$, where $c > 1$ is an absolute constant.
We call such a set $A$ a \emph{$c$-ratio cover} of $K$ (Definition~\ref{def:ratio-cover}).

In Section~\ref{sec:domin-subs-conv}, we % provide two different proofs
show
that convex sets admit small ratio covers.
In particular, we establish the following optimal cardinality estimate for ratio covers of compact convex sets, which is one of the main contributions of this work.
We expect this covering result to be of independent interest, beyond its application to estimation of mixture distributions.

\begin{framed}
  \begin{enumerate}
  \item[3.] We show that there exist universal constants $C_1, C_2>1$ such that the following holds: for every $d \geq 1$, every convex compact subset of $\R_+^d$ admits a $C_1$-ratio cover with at most $C_2^d$ elements (\cref{thm:ratio-cover}).
  \end{enumerate}
\end{framed}

With these ingredients in hand, we prove our high-probability guarantees for model aggregation (Section~\ref{sec:model-agg}; using the first and third ingredients) and convex aggregation (Section~\ref{sec:mixture-models}; using all three ingredients).

\subsection{Further applications of ratio covers}
Theorem~\ref{thm:ratio-cover} also appears useful beyond our two main estimation results. We outline a few further applications here, and develop them in detail in Section~\ref{subsection:apply-pareto} and Section~\ref{sec:apps}.

\paragraph{Approximate Pareto curves.}
Let $K \subset \R_+^d$ represent the set of attainable values in a multi-objective optimization problem. For instance, a point $x=(x_1,\dots,x_d)\in K$ could represent producing $x_1$ units of product~1, $x_2$ units of product~2, and so on. In general, the Pareto-optimal set may be infinite. The influential work of Papadimitriou and Yannakakis \cite{papadimitriou2000approximability} shows that there exists a small subset of $K$ that \textit{approximately} dominates all points in $K$. In our terminology, a set that $(1+\eps)$-dominates $K$ is exactly a $(1+\eps)$-ratio cover. Their bounds for general $K$ depend on $d,\eps$ and on the range parameter
$R = \max_{i \in [d]}\frac{\max_{x \in K} x_i}{\min_{y \in K} y_i}$, where here and in what follows $[d] = \{1, \ldots, d\}$.
In Section~\ref{subsection:apply-pareto}, we use our ratio cover machinery (under convexity of $K$) to remove the dependence on $R$. We expect this to be useful in settings where $K$ is convex, for example when $K$ is the feasible region of a linear program.

\paragraph{Expected error guarantees via sharpening Yang--Barron.}
As remarked earlier, Yang and Barron provide a general recipe for converting KL covering numbers into in-expectation error bounds. While very general, their method often introduces an unnecessary logarithmic factor. In Section~\ref{sec:yangbarron-sharp}, we sharpen the Yang--Barron guarantee in settings where the class is convex and amenable to our ratio-cover approach. The plan is: (i) obtain an estimator with squared Hellinger error $\eps$ (often via entropy methods), (ii) use ratio covers to build a KL cover of the local Hellinger ball, and (iii) run the Yang--Barron estimator using this local KL cover. The key improvement is that we only need a KL cover \emph{locally}, and the ratio-cover structure remains helpful even when the initial Hellinger estimator occasionally lands outside the target $\eps$-ball (thanks to \cref{lemma:kl-from-ball}).
This yields our in-expectation guarantee for convex aggregation, scaling as ${\M}/{n}$ up to constants.

\paragraph{KL covering numbers.}
The recent PhD thesis ``Divergence Coverings'' of Tang \cite{tang2022divergence} develops techniques for constructing KL covers and highlights the problem of bounding the KL covering number of the $d$-dimensional simplex. In Section~\ref{subsection:apply-kl-cover}, we review this line of work and show how ratio covers sharpen the KL covering number for the simplex. We emphasize that this section is mainly an illustration of how our techniques can sharpen recent results: in fact, our approach extends immediately beyond the simplex to the convex hull of an arbitrary dictionary of $d$ distributions, and can be used to control corresponding local covering numbers.

\subsection{Related literature}
The literature on aggregation under logarithmic loss is vast; we only highlight lines of work that are closest to our setting.

\paragraph{Sequential prediction, logarithmic loss, and universal coding.}
Sequential prediction with logarithmic loss lies at the intersection of online learning, universal coding, and Bayesian mixture methods.
In universal coding, predictors such as the Krichevsky--Trofimov estimator yield sharp nonasymptotic redundancy guarantees under log-loss \cite{krichevsky1981performance}.
In online learning, log-loss is \emph{mixable} in the sense of Vovk \cite{vovk1990aggregating}, which leads to logarithmic regret for exponential-weights predictors on the simplex; see, e.g., \cite{vovk1990aggregating,vovk1998mixability,cesabianchi2006plg,hazan2007logarithmic}.
Universal portfolio selection \cite{cover1991universal} provides another prominent instance: cumulative log-wealth is formally equivalent to log-loss prediction over the simplex.
Parts of our analysis are closest to recent work converting sequential guarantees into high-probability statistical risk bounds for exp-concave losses via online-to-batch arguments \cite{vanderhoeven2023high}.

\paragraph{Statistical aggregation and progressive mixtures.}
In model aggregation for density estimation, \emph{progressive mixture} rules were introduced and analyzed in \cite{yang2000mixing}, with a closely related information-theoretic perspective in \cite{yang1999information}.
These developments connect to earlier work on Bayes rules and information consistency \cite{barron1987bayes} and to mixture-based approaches to model aggregation \cite{catoni1997mixture}, as well as to online viewpoints on aggregating strategies \cite{desantis1988learning,vovk1990aggregating}.
Mirror averaging, closely related to progressive mixtures, yields oracle inequalities in several aggregation setups (primarily for quadratic loss) \cite{juditsky2008learning}.
The broader aggregation framework and minimax formulations of optimal rates were developed in \cite{nemirovski2000topics,tsybakov2003optimal}.
For deviation bounds, progressive mixtures can be suboptimal in model aggregation settings (shown for square loss in \cite{audibert2007progressive}), motivating alternative constructions not directly driven by regret; in particular, deviation-optimal aggregation for bounded regression losses has been obtained via $Q$-aggregation \cite{dai2012q-aggregation,lecue2014q-aggregation,bellec2017optimal,mourtada2023local}.

\paragraph{Likelihood methods, entropy, and geometry.}
A central tool in nonparametric density estimation is the control of likelihood-based procedures through metric entropy and likelihood-ratio inequalities, as in sieve MLE analyses \cite{WongShen1995} and model-selection/concentration methods \cite{massart2007concentration}.
Estimator selection under Hellinger-type risks and test-based constructions has been developed in \cite{baraud2011estimator} and refined in \cite{baraud2017new,baraud2018rho}.
Closer to our focus, local (bracketing) geometric properties of finite mixture models are studied in \cite{gassiat2014local} for finite location mixtures of a single base density under smoothness assumptions.
In the special case of mixture density estimation over a finite dictionary, optimal KL aggregation rates for the MLE under boundedness assumptions on the dictionary are investigated in \cite{li1999mixture,li1999estimation,rakhlin2005risk, Rigollet2012KLAggregation, ButuceaDelmasDutfoyFischer2017Optimal, dalalyan2018optimal}.
By contrast, our results apply to arbitrary dictionaries of densities without imposing additional structural conditions.

\subsection{Notation and preliminary results}
Let $[k]$ denote $\{1, 2, \ldots, k\}$. $\wh L$ is the normalized negative log-likelihood, namely
\begin{equation}
  \label{eq:def-log-likelihood}
  \wh L(p) = -\frac{1}{n}\sum\limits_{i = 1}^n\log p(X_i).
\end{equation}
Squared Hellinger divergence is denoted as
\begin{equation}
  \label{eq:def-hellinger}
  \hels(p,q) \triangleq \frac{1}{2} \int_{\mathcal X} \big(\sqrt{p(x)}-\sqrt{q(x)} \big)^2\di \mu(x).
\end{equation}
We will use the following relation between KL divergence and the Hellinger distance.
\begin{lemma}
    \label{lem:kltohellinger}
Assume that $p,q$ are two densities on $\X$ with respect to $\mu$.
It holds that
\[
\kl(p,q) \le \frac{2}{(\sqrt{e}-1)^2}\,\hels(p,q)\,
\max\left\{1,\ \sup_{x\in\X}\log\left(\frac{p(x)}{q(x)}\right)\right\}.
\]
\end{lemma}

For the sake of completeness we present the proof of this result in \Cref{sec:remainingproofs}. We note that similar identities appear in \citep[Lemma 5]{birge1998minimum}, \cite[Lemma 4.4]{birge1983approximation} with the same constants appearing in \cite[Fact 4]{tang2022divergence}, but for a definition of the Hellinger distance without $1/2$ in front. See also \cite{birge1993rates}.

Our next preliminary result is a corollary of Theorem 1 in \cite{vanderhoeven2023high}, which is stated there for more general exponentially concave losses. We present this result in the form required for our future analysis.

\begin{lemma}[Model aggregation with bounded density ratios]
  \label{lem:onlinetobatch}
  Assume $X_{1},\ldots,X_{n}$ are
\iid random variables with arbitrary and unknown density $p^\star$ on $(\X, \mu)$.
Let $p_{1},\ldots,p_{\M}$ be densities such that for some $m>0$,
\[
\sup_{x \in \X}\max_{i,j\in[M]}\left|\log\left(\frac{p_i(x)}{p_j(x)}\right)\right|\le m.
\]
There is an estimator $\widehat{p}$ such that for every $\delta\in(0,1)$, with probability at least $1-\delta$,
\begin{equation}
\label{eq:main-theorem-bound}
\kll{p^\star}{\widehat{p}} \le \min_{j \in [M]}\kll{p^\star}{p_j} + \frac{2\log\left(M\right)+\frac{25}{4}\left(e-2\right)\left(m+1\right)\log\left(\frac{1}{\delta}\right)}{n} .
\end{equation}
\end{lemma}

Exact construction of the estimator as well as the self-contained proof of \Cref{lem:onlinetobatch} are presented in \Cref{sec:highprobabilityonlinetobatch}. As a short comment, the estimator $\widehat{p}$ is a specific variance-adjusted version of the standard progressive mixture rule. Note that \Cref{lem:onlinetobatch}
does not solve the model aggregation problem introduced above as it requires that the density ratios are bounded, an assumption we aim to avoid in this work.

\subsection{Structure of the paper}
In \Cref{sec:model-agg} we study model aggregation over a finite family of densities and prove \Cref{thm:modelaggregation}. 
In \Cref{sec:local-cover-numb}, we establish a sharp local covering bound in Hellinger distance for mixtures (\Cref{thm:local-entropy-mixtures}).
\Cref{sec:domin-subs-conv} is devoted to ratio covers, including the proof of \Cref{thm:ratio-cover} and additional results on approximate Pareto curves.
These tools are combined in \Cref{sec:mixture-models} to obtain the high-probability bound for convex aggregation (\Cref{thm:convexaggregation}). 
Further applications of ratio covers are presented in \Cref{sec:apps}: expected KL guarantees via a refined Yang--Barron argument (\Cref{sec:yangbarron-sharp}) and KL covering numbers (\Cref{subsection:apply-kl-cover}). 
We conclude with the misspecified setting in \Cref{sec:misspecified}. 

The appendices contain a proof of the high-probability online-to-batch conversion (\Cref{sec:highprobabilityonlinetobatch}), as well as lower bounds illustrating the suboptimality of classical estimators (\Cref{sec:suboptimal}). 
In particular, we obtain matching upper and lower bounds on the tail behavior of Bayesian model averaging in the case of two densities, which falls short of the ideal deviation bound.
Finally, we include a detailed presentation of the Birg\'{e}--Le Cam tournament (\Cref{sec:blc-local-pointwise}).

\section{Model aggregation}\label{sec:model-agg}

In this section, we provide a complete analysis of model aggregation over $\M$ individual densities, where no assumptions are made on the individual densities.
Despite its novelty, this result is technically less involved than the mixture case and can be seen as a preliminary building block introducing ideas that will also appear in this more complex setting.
In particular, even at the preliminary stage, we do not need to use the Birg\'{e}-Le Cam tournaments, but rather exploit simple properties of the MLE in this case.

\begin{theorem}
\label{thm:modelaggregation}
    Let $X_{1}, \ldots, X_n$ be an i.i.d.\ sample from some unknown density $p^\star$ that belongs to a known class of arbitrary densities $\mathcal{P} = \{p_1, \ldots, p_\M\}$ defined on the space $\X$ with respect to the measure $\mu$. There is a density estimator $\widehat{p}$ such that, with probability at least $1 - \delta$,
    \[
    \kl\left(p^\star, \widehat{p}\right) \le \min\left\{\frac{75\log(2M)\log(2/\delta)}{n}, \log(M)\right\}.
    \]
This bound is optimal, in the sense that there are $\M$ distributions such that no estimator can perform better up to universal constant factors, as shown in \cite[Lemma~1]{mourtada2025estimation} and~\cite[Theorem~11]{van2025nearly}. 
\end{theorem}
For the sake of presentation we assume that we are given a sample of size $2n$. In general, provided that the sample size is at least $2$ we can split the sample into two equally sized independent parts, while dropping at most one element of the original sample.
This will only lead to a larger absolute constant in the final statement. We describe the estimator $\widehat{p}$ as follows:
\begin{framed}
\begin{itemize}
\item Use the first $n$ observations (defining $\wh L$) to construct a data-dependent set of almost empirical risk minimizers
\[
\wh \Qs_1 = \left\{ p \in \mathcal P : \wh L (p) - \min_{q \in \mathcal{P}} \wh L (q) \leq \frac{2 \log (4\M/\delta)}{n} \right\}.
\]
\item Construct a (random) density class 
\begin{equation}
\label{eq:wprimeset}
\wh \Qs_2 = \bigg\{\frac{1}{2}p + \frac{1}{2|\wh \Qs_1|}\sum\limits_{q \in \wh \Qs_1}q: p \in \wh \Qs_1\bigg\}.
\end{equation}
\item Using the second half of the sample, run the density estimator of \Cref{lem:onlinetobatch} over $\wh \Qs_2$ to output the final density $\widehat{p}$. Finally, if $75\log(2M)\log(2/\delta) \ge 2n\log(M)$ output instead $\widehat{p} = \frac{1}{M}\sum\limits_{p \in \mathcal{P}}p$.
\end{itemize}
\end{framed}

Let us quickly explain why the procedure works. We first show that the construction of the set $\wh \Qs_1$ of almost empirical likelihood maximizers provides a set of small Hellinger radius that contains $p^\star$. However, closeness in Hellinger radius does not imply closeness in the Kullback--Leibler divergence.
This is where the smoothed class $\wh \Qs_2$ helps: the construction does not inflate the Hellinger distance, but it allows one to control the density ratio. However, as an inevitable price of smoothing $p^\star$ does not have to belong to $\wh \Qs_2$.
This is where now the regret-based techniques of Lemma~\ref{lem:onlinetobatch} help, since they provide both the optimal tail bounds and allow to bypass the problem $p^\star \notin \wh \Qs_2$. Remarkably, the final estimator is optimal up to multiplicative constant factors.

The proof of Theorem \ref{thm:modelaggregation} requires a preliminary lemma. The first result is standard and similar derivations can be found in \eg \cite{WongShen1995}. 

\begin{lemma}
\label{lem:almostmle}
Let $\mathcal{P}=\{p_1,\ldots,p_\M\}$ be a finite set of densities with respect to a measure $\mu$ on $\mathcal{X}$.
Fix $p^\star\in\mathcal{P}$ and let $X_1,\ldots,X_n$ be i.i.d.\ distributed according to $p^\star$.
For $\delta \in (0, 1)$, set
\[
\wh \Qs_1
=\bigg\{p\in\mathcal{P}:\ \wh L(p)-\min_{q\in\mathcal{P}}\wh L(q)\le\frac{2\log(2\M/\delta)}{n}\bigg\}.
\]
Then, with probability at least $1-\delta$, one has $p^\star \in \wh \Qs_1$ and
\[
\max_{p \in \wh \Qs_1}\hel^2(p^\star,p) \le\frac{2\log(2\M/\delta)}{n}.
\]
\end{lemma}

\begin{proof}
Fix $t>0$. For any $p\in\mathcal{P}$, by Markov's inequality,
\begin{align*}
\P\left(\wh L(p)-\wh L(p^\star)\le t\right)
&=\P\left(\sqrt{\frac{\prod_{i=1}^n p(X_i)}{\prod_{i=1}^n p^\star(X_i)}}\ge \exp(-nt/2)\right)\\
&\le \exp(nt/2)\,\E\left[\sqrt{\frac{\prod_{i=1}^n p(X_i)}{\prod_{i=1}^n p^\star(X_i)}}\right]\\
&= \exp(nt/2)\,\big(1-\hel^2(p^\star,p)\big)^{n}
\le\exp\Big(\tfrac{nt}{2}-n\hel^2(p^\star,p)\Big),
\end{align*}
where we used $1-u\le \exp(-u)$. Let $\mathcal{P}_t=\{p\in\mathcal{P}:\hel^2(p^\star,p)\ge t\}$. By a union bound,
\[
\P\left(\exists p\in\mathcal{P}_t:\ \wh L(p)-\wh L(p^\star)\le t\right)
\le\sum_{p\in\mathcal{P}_t} \exp(-nt/2)
\le M\,\exp(-nt/2).
\]
Note that for any $p\in\mathcal{P}$,
$\wh L(p)-\min_{q\in\mathcal{P}}\wh L(q)\le t$ implies $\wh L(p)-\wh L(p^\star)\le t$,
since $\min_{q \in \mathcal{P}}\wh L(q)\le \wh L(p^\star)$. Therefore,
\[
\P\left(\exists p\in\mathcal{P}: \wh L(p)-\min_{q \in \mathcal{P}} \wh L(q)\le t\ \text{and}\ \hel^2(p^\star,p)\ge t\right) \le M \exp(-nt/2).
\]
For the second part, apply the same argument with $-t$ in place of $t$. For any $p\in\mathcal{P}$, it holds that
\begin{align*}
\P\left(\wh L(p)-\wh L(p^\star)\le -t\right)
&=\P\left(\sqrt{\frac{\prod_{i=1}^n p(X_i)}{\prod_{i=1}^n p^\star(X_i)}}\ge \exp(nt/2)\right)
\\
&\le \exp(-nt/2)\,\Big(\int\sqrt{p(x)p^\star(x)}\,d\mu(x)\Big)^{n}
\le \exp(-nt/2).
\end{align*}
By a union bound,
\[
\P\left(\min_{q\in\mathcal{P}} \wh L(q) \le \wh L(p^\star)-t\right)
\le \sum_{p\in\mathcal{P}} \P\left(\wh L(p)-\wh L(p^\star)\le -t\right)
\le M \exp(-nt/2).
\]
Equivalently, we have
\[
\P\left(\wh L(p^\star)-\min_{q\in\mathcal{P}} \wh L(q) \geq t\right)\le M \exp(-nt/2).
\]
Choose $t={2\log(2\M/\delta)}/{n}$. Then each of the two bad events above has probability at most $\delta/2$, so by a union bound they fail simultaneously with probability at most $\delta$. Hence, with probability at least $1-\delta$, every $p\in\wh \Qs_1$ satisfies $\hels(p^\star,p) \leq t$, and also $\wh L(p^\star)-\min_{q\in\mathcal{P}} \wh L(q)\le t$ hence $p^\star\in \wh \Qs_1$. This yields the claim.
\end{proof}

\paragraph{Proof of Theorem \ref{thm:modelaggregation}.}
Fix $\delta \in (0, 1)$. By Lemma \ref{lem:almostmle}, with probability at least $1-\delta/2$, the set
\[
\wh \Qs_1 = \left\{ p \in \mathcal P : \wh L (p) - \min_{q \in \mathcal{P}} \wh L (q) \leq \frac{2 \log (4\M/\delta)}{n} \right\}
\]
contains $p^\star$ and has the Hellinger diameter at most $2\sqrt{{2\log(4\M/\delta)}/{n}}$. From now on we work on this event and focus on the second half of the sample. Using the second half of the sample we output the density $\widehat{p}$ obtained using the procedure of \Cref{lem:onlinetobatch} with respect to $\wh \Qs_2$. 
Let
\begin{equation*}
  \widetilde{p} = \frac{1}{2}p^\star + \frac{1}{2|\wh \Qs_1|}\sum\limits_{q \in \wh \Qs_1}q.
\end{equation*}
Note that $\widetilde{p} \in \wh \Qs_2$ given by \eqref{eq:wprimeset}. Since we aim to use \Cref{lem:onlinetobatch} to bound $\kll{p^\star}{\widehat{p}}$, we will first verify its condition on the densities. 
Observe that by construction for each $p, q \in \wh \Qs_2$ we have
\[
\sup_{p,q \in \wh \Qs_2}\sup_{x\in\X}\log\left(\frac{p(x)}{q(x)}\right) \le \log\bigl(1 + |\wh \Qs_1|\bigr) \le \log(2M).
\]
Therefore, by  \Cref{lem:onlinetobatch} we have, with probability at least $1 - \delta/2$,
\begin{equation}
\kll{p^\star}{\widehat{p}} \le \kll{p^\star}{\widetilde{p}} + \frac{2\log(M) + \frac{25}{4}\left(e-2\right)\left(\log(2M)+1\right)\log\left({2}/{\delta}\right)}{n}.\label{eq:model-kl}
\end{equation}
Finally, we have by Lemma \ref{lem:kltohellinger} and the convexity of the Hellinger distance squared, observing that the ratio of $p^\star$ and $\frac{1}{2}p^\star +\frac{1}{2|\wh \Qs_1|}\sum\limits_{q \in \wh \Qs_1}q$ is bounded by $2$,
\begin{align*}
\kl(p^\star, \widetilde{p}) &\le \frac{2\max\{1, \log(2)\}}{(\sqrt{e}-1)^2}\hels\left(p^\star, \frac{1}{2}p^\star + \frac{1}{2|\wh \Qs_1|}\sum\limits_{q \in \wh \Qs_1}q\right)
\\
&\le \frac{1}{(\sqrt{e}-1)^2}\hels\left(p^\star, \frac{1}{|\wh \Qs_1|}\sum\limits_{q \in \wh \Qs_1}q\right) \le \frac{2\log (4\M/\delta)}{(\sqrt{e}-1)^2n}.
\end{align*}
Plugging this into \eqref{eq:model-kl} and using the union bound, we have with probability at least $1 - \delta$,
\begin{align*}
\kl(p^\star, \widehat{p}) &\le \frac{2\log(M) + \frac{25}{4}\left(e-2\right)\left(\log(2M)+1\right)\log\left(\frac{2}{\delta}\right)}{n} + \frac{2\log (4\M/\delta)}{(\sqrt{e}-1)^2n}
\\
&\le \frac{25\log(2M)\log(2/\delta)}{n}.
\end{align*}
Note that we used a sample of size $2n$; for any original sample size at least $2$, splitting into two equal parts (dropping at most one observation) increases the absolute constant by at most a factor of $3$ (the worst case is a total sample size of $3$). Finally, if $75\log(2M)\log(2/\delta) \ge 2n\log(M)$, we output $\widehat{p} = \frac{1}{M}\sum\limits_{p \in \mathcal{P}}p$, which guarantees $\kl(p^\star, \widehat{p}) \le \log(M)$. The claim follows. \qed

\section{Local Hellinger entropy for mixture models
}
\label{sec:local-cover-numb}

In this section we study the local Hellinger geometry of convex hulls of densities. This result already suffices to establish the optimal rate for
estimation of mixture distributions under Hellinger loss.
By contrast, the best previously known statistical rates incur additional logarithmic factors due to reliance on VC-type covering arguments; see, for instance, \cite[Proposition~7 and Corollary~3]{baraud2018rho}.
Our estimation guarantee in Hellinger distance is given by the bound \eqref{eq:optestimationinhellinger} below.

We first introduce notation. Let $\M \geq 2$ and, as above, let $\mathcal{P} = \{p_1, \dots, p_\M\}$ be a set of densities on $\mathcal{X}$ with respect to $\mu$.
Given a set of densities $\mathcal P$, define its local Hellinger entropy by
\begin{equation}
\label{eq:localentropy}
\nloc(\mathcal{P},\eps)
=
\sup\limits_{p^\star \in \mathcal P}\sup_{\eta\ge \eps}\ N_{\hel}\big(B_{\hel}(p^\star,\eta)\cap\mathcal{P},\eta/2\big),
\end{equation}
where $B_{\hel}(p,r)=\{q \text{ is a density on } (\X, \mu):\hel(p,q)\le r\}$ and $N_{\hel}(\mathcal{F},\eps)$ is the covering number of $\mathcal{F}$ in the Hellinger metric $\hel$ at scale $\eps$. For any $\theta = (\theta_1, \dots, \theta_\M) \in \Delta_{\M-1}$, where in what follows
\[
\Delta_{\M-1} = \bigg\{\theta \in \R^M_{+}: \sum_{j = 1}^M \theta_{j} = 1 \bigg\},
\]
we define $p_\theta = \sum_{j=1}^M \theta_j p_j$. It is convenient to define the pseudo-metric $d_{\hel}$ on the simplex $\Delta_{\M-1}$, which is the Hellinger distance between mixtures, namely for any $\theta, \theta' \in \Delta_{\M-1}$, one has $d_{\hel} (\theta, \theta') = \hel (p_\theta, p_{\theta'})$. The following result provides a universal upper bound on local Hellinger covering numbers of the class of mixtures of $\M$ densities.

\begin{theorem}
  \label{thm:local-entropy-mixtures}
  For any finite dictionary of densities $\mathcal{P}=\{p_1,\dots,p_\M\}$ on $\mathcal X$ with respect to the measure $\mu$ and any $\varepsilon>0$, one has
  \[
  \nloc\big(\conv(\mathcal P),\varepsilon\big)\le 64^{\M-1}.
  \]
\end{theorem}

\begin{proof}
Fix $p^\star\in\conv(\mathcal P)$ and $\eta\ge\varepsilon$, and choose $\theta_0\in\Delta_{\M-1}$ such that $p^\star=p_{\theta_0}$. Set
\[
\mathcal{B} \triangleq \{\theta\in\Delta_{\M-1}: d_{\hel}(\theta,\theta_0)\le \eta\}.
\]
Then
\[
B_{\hel}(p^\star,\eta)\cap\conv(\mathcal P)=\{p_\theta:\theta\in \mathcal{B}\},
\]
and moreover, for $A\subseteq\Delta_{\M-1}$ and $\delta>0$, let $N_{d_{\hel}}(A,\delta)$ denote the covering number of $A$ in the pseudo-metric $d_{\hel}$ at scale $\delta$. Then
\begin{equation}\label{eq:cover-transfer}
N_{\hel}\big(B_{\hel}(p^\star,\eta)\cap\conv(\mathcal P),\eta/2\big)\le N_{d_{\hel}}(\mathcal{B},\eta/2).
\end{equation}
Since $d_{\hel}\le 1$, if $\eta>1$ then $\mathcal{B}=\Delta_{\M-1}$ and by monotonicity in the radius $N_{d_{\hel}}(\mathcal{B},\eta/2)\le N_{d_{\hel}}(\Delta_{\M-1},1/2)$, so it suffices to prove the bound below for $0<\eta\le 1$.

Assume $0<\eta\le 1$. By definition, $\mathcal{B}=\{\theta\in\Delta_{\M-1}:\hels(p_{\theta_0},p_\theta)\le \eta^2\}$; the map $\theta\mapsto p_\theta$ is affine and $q\mapsto \hels(p_{\theta_0},q)$ is continuous and convex, so $\mathcal{B}$ is convex and compact. Moreover, for any $\theta,\theta'\in\Delta_{\M-1}$ and any $t\in[0,1]$, convexity of $\hels(p_\theta,\cdot)$ gives
\[
\hels\left(p_\theta,p_{(1-t)\theta+t\theta'}\right)
=\hels\left(p_\theta,(1-t)p_\theta+tp_{\theta'}\right)\le t\hels(p_\theta,p_{\theta'}),
\]
and therefore
\begin{equation}\label{eq:convex-shrink}
d_{\hel}\big(\theta,(1-t)\theta+t\theta'\big)\le \sqrt{t}\,d_{\hel}(\theta,\theta').
\end{equation}
Using $d_{\hel}\le 1$ and \eqref{eq:convex-shrink} with $\theta=\theta_0$ shows that $(1-\eta^2)\theta_0+\eta^2\Delta_{\M-1}\subseteq \mathcal{B}$, hence (with $\vol_{\M-1}$ the Lebesgue measure on the affine hull $\{\theta\in\R^M:\sum_{j=1}^M\theta_j=1\}$) one has
\[
\vol_{\M-1}(\mathcal{B})\ge \vol_{\M-1}\big((1-\eta^2)\theta_0+\eta^2\Delta_{\M-1}\big)=\eta^{2(M-1)}\vol_{\M-1}(\Delta_{\M-1})>0.
\]

Let $\{\theta_1,\dots,\theta_N\}\subset \mathcal{B}$ be a maximal $(\eta/2)$-separated set in $(\mathcal{B},d_{\hel})$. By the standard property we have $N_{d_{\hel}}(\mathcal{B},\eta/2)\le N$. Fix $\lambda\triangleq 1/64$ and define $A_i\triangleq (1-\lambda)\theta_i+\lambda \mathcal{B}$ for $i=1,\dots,N$. Since $\mathcal{B}$ is convex and $\theta_i\in \mathcal{B}$, we have $A_i\subseteq \mathcal{B}$. If $\theta\in A_i$ and $\theta'\in A_j$ with $i\neq j$, write $\theta=(1-\lambda)\theta_i+\lambda\tilde\theta$ and $\theta'=(1-\lambda)\theta_j+\lambda\tilde\theta'$ for some $\tilde\theta,\tilde\theta'\in \mathcal{B}$; then \eqref{eq:convex-shrink} with $t=\lambda$ and the triangle inequality yield
\[
d_{\hel}(\theta_i,\theta)\le \sqrt{\lambda}\,d_{\hel}(\theta_i,\tilde\theta)
\le \sqrt{\lambda}\big(d_{\hel}(\theta_i,\theta_0)+d_{\hel}(\theta_0,\tilde\theta)\big)\le 2\eta\sqrt{\lambda},
\]
and similarly $d_{\hel}(\theta_j,\theta')\le 2\eta\sqrt{\lambda}$. Therefore,
\[
d_{\hel}(\theta,\theta')\ge d_{\hel}(\theta_i,\theta_j)-d_{\hel}(\theta_i,\theta)-d_{\hel}(\theta_j,\theta')
> \frac{\eta}{2}-4\eta\sqrt{\lambda}=0,
\]
since $\sqrt{\lambda}=1/8$; hence $A_i\cap A_j=\emptyset$. Observe that $\vol_{\M-1}(A_i)=\lambda^{\M-1}\vol_{\M-1}(\mathcal{B})$. As the sets $A_1,\dots,A_N$ are pairwise disjoint and contained in $\mathcal{B}$,
\[
\vol_{\M-1}(\mathcal{B})\ge \sum_{i=1}^N \vol_{\M-1}(A_i)=N\,\lambda^{\M-1}\vol_{\M-1}(\mathcal{B}),
\]
and since $\vol_{\M-1}(\mathcal{B})>0$ we conclude $N\le \lambda^{-(M-1)}=64^{\M-1}$. Thus $N_{d_{\hel}}(\mathcal{B},\eta/2)\le 64^{\M-1}$, and combining with \eqref{eq:cover-transfer} gives
\[
N_{\hel}\big(B_{\hel}(p^\star,\eta)\cap\conv(\mathcal P),\eta/2\big)\le 64^{\M-1}
\quad\text{for all }p^\star\in\conv(\mathcal P),\ \eta\ge\varepsilon.
\]
Taking the suprema in \eqref{eq:localentropy} yields $\nloc(\conv(\mathcal P),\varepsilon)\le 64^{\M-1}$.
\end{proof}

\section{Ratio covers of convex sets}
\label{sec:domin-subs-conv}
After approximating the target distribution in Hellinger distance (as we did above, for instance in
the special case of finite mixtures),
a central question in this work is how to convert this control into a \emph{small} set of candidate distributions such that at least one is close to $p^\star$ in KL divergence. We recall from Lemma~\ref{lem:kltohellinger} that the KL divergence $\kll{p}{q}$ can be bounded in terms of the Hellinger distance together with a factor depending on the maximal density ratio $\sup_{x\in\mathcal X}\frac{p(x)}{q(x)}$. In particular, if $p$ and $q$ lie in the same squared Hellinger ball of radius $\eps$ and the density ratio is bounded by a constant, then $\kll{p}{q}\le C\,\eps$ for some constant $C>0$ depending only on that ratio bound. This motivates the general notion of a \textit{ratio cover}, which we will later apply to Hellinger balls to obtain the KL covers needed in our applications.

\begin{definition}
  \label{def:ratio-cover}
  Let $K$ be a compact subset of $\R_+^d$ and $\alpha \geq 1$.
  We say that a subset $A \subset K$ is an \emph{$\alpha$-ratio cover} of $K$ if for every $\theta = (\theta_1, \dots, \theta_d) \in K$, there exists an element $\phi = (\phi_1, \dots, \phi_d) \in A$ such that $\theta_j \leq \alpha \, \phi_j$ for every $j=1, \dots, d$.
\end{definition}

Our goal is to show the existence of $C_1$-ratio covers of size $(C_2)^d$ for convex sets $K$, with constants $C_1,C_2\ge 1$. In our proof of Theorem~\ref{thm:convexaggregation} for mixture models, we will ultimately pay a term proportional to $\log(|A|)/n$ for a ratio cover of size $|A|$: hence even a ratio cover of size $|A|=(\log d)^d$ would yield a suboptimal upper bound featuring an additional $\log \log d$ factor, and similarly any super-constant ratio $\alpha$ would yield suboptimal results (both issues are also important for our other applications). With this motivation, we state one of our main results.

\begin{theorem}
\label{thm:ratio-cover}
For every $d \geq 1$ and every convex and compact set
$K \subset \R_+^d$, there exists a subset $A \subset K$
with at most $2^{8d}$ elements that is a $32$-ratio cover of $K$.
\end{theorem}

In the remainder of this section, we first briefly comment on the theorem statement, and then we discuss an idea of the proof of Theorem~\ref{thm:ratio-cover}. In Subsection~\ref{subsection:ratio-proof}, we prove Theorem~\ref{thm:ratio-cover}.

First, the constraint $A\subset K$ is essential: without it, there is a trivial ``cover'' obtained by choosing the single vector
$\phi\in\R_+^d$ with coordinates $\phi_j=\max_{\theta\in K}\theta_j$, which need not belong to $K$ and is not useful for our applications. Second, the definition must be one-sided. Indeed, if one required the two-sided condition $\phi_j/\alpha \le \theta_j \le \alpha\phi_j$, then even in dimension $d=1$ the set $K=[1,L]$ with $L\ge 2$ would require at least $\log_\alpha L$ points. Third, convexity is required: for $d=2$, the nonconvex set
\[
K=\set{(x,1/x):\ 1/L\le x\le L}
\]
also requires at least $\log_\alpha L$ points for any fixed $\alpha$.
(Of course, one could drop the convexity assumption on $K$ but then relax the requirement to $A \subset \conv (K)$.)
Finally, even for convex $K$, an exponential dependence on $d$ is sometimes unavoidable for constant $\alpha$.
A simple illustration is $\Delta_1=\{(t,1-t):t\in[0,1]\}\subset\R_+^2$: if $\alpha = 1.9$, no single point $\phi=(a,1-a)\in\Delta_1$ can $\alpha$-ratio cover both extreme points $(1,0)$ and $(0,1)$. Hence any $1.9$-ratio cover of $\Delta_1$ must have size at least $2$. Consequently, if $d$ is even, a $1.9$-ratio cover for the product set $\prod_{i=1}^{d/2}\Delta_1\subset\R_+^d$ must have size at least $2^{d/2}$: each block $\Delta_1$ contributes an independent binary choice, so one needs a distinct cover point for each of the $2^{d/2}$ combinations.

We provide two different approaches to the proof of \Cref{thm:ratio-cover}. The first is somewhat longer, but fully constructive and deterministic, based on first principles and exploiting the convexity of the set. The second proof, presented in \Cref{subsection:epsilonnetproof}, is shorter and relies on a reduction to $\varepsilon$-net arguments presented in \cite{diakonikolas2010small} together with the volumetric argument used in the proof of \Cref{thm:local-entropy-mixtures}.

\paragraph{Proof roadmap for \Cref{thm:ratio-cover}.}
We outline the main ideas of the proof. We begin with two warm-ups that highlight the basic geometry behind ratio covers, and then explain how the full argument combines discretization, averaging, and induction to obtain a $C_1$-ratio cover of size $C_2^d$.

\textbf{Warm-up 1: a $d$-ratio cover with a single point.}
For each coordinate $j\in[d]$, let $m_j=\max_{\theta\in K}\theta_j$ (the maximum exists since $K$ is compact), and pick $\overline{\theta}^{(j)}\in K$ such that $\overline{\theta}^{(j)}_j=m_j$.
Define
\[
\phi^{(1)} \triangleq \frac{1}{d}\sum_{j=1}^d \overline{\theta}^{(j)} \in K,
\]
where $\phi^{(1)}\in K$ by convexity.
Then for any $\theta\in K$ and any $i\in[d]$,
\[
\theta_i \le m_i = \overline{\theta}^{(i)}_i
\le \sum_{j=1}^d \overline{\theta}^{(j)}_i
= d\,\phi^{(1)}_i.
\]
Thus $\{\phi^{(1)}\}$ is a $d$-ratio cover of $K$.

\textbf{Warm-up 2: a $4$-ratio cover with $|A|\le (\lceil \log_2 d\rceil+1)^d$.}
Warm-up~1 is too coarse since the ratio factor grows with $d$.
To reduce the ratio, we discretize each coordinate of $\theta$ into dyadic bins relative to $\phi^{(1)}$.
Fix $\theta\in K$. If $\theta_i\le \phi^{(1)}_i$, then coordinate $i$ is already $1$-covered by $\phi^{(1)}$.
Otherwise, $\theta_i>\phi^{(1)}_i$, and the bound from Warm-up~1 implies $\theta_i\le d\,\phi^{(1)}_i$.
(If $\phi^{(1)}_i=0$, then necessarily $m_i=0$ and hence $\theta_i=0$ for all $\theta\in K$, so only the first case can occur.)
Therefore each coordinate $i$ must satisfy either
\[
\theta_i \le \phi^{(1)}_i,
\quad\text{or}\quad
\theta_i \in (2^j\phi^{(1)}_i,\,2^{j+1}\phi^{(1)}_i]
\ \text{for some }j\in\{0,\dots,\lceil\log_2 d\rceil-1\}.
\]
This yields at most $(\lceil\log_2 d\rceil+1)^d$ possible bin patterns across all coordinates.
Let $A_2\subset K$ contain one representative $\phi^{(2)}$ for each bin pattern that occurs among points of $K$, and define
\[
A \triangleq \frac{\phi^{(1)}+A_2}{2}
= \left\{\frac{\phi^{(1)}+\phi^{(2)}}{2}:\ \phi^{(2)}\in A_2\right\}.
\]
Given $\theta\in K$, pick $\phi^{(2)}\in A_2$ with the same bin pattern.
If $\theta_i\le \phi^{(1)}_i$, then
\[
\theta_i \le \phi^{(1)}_i \le \phi^{(1)}_i+\phi^{(2)}_i
= 2\frac{\phi^{(1)}_i+\phi^{(2)}_i}{2}.
\]
Otherwise, for some $j$ we have $\theta_i\in(2^j\phi^{(1)}_i,2^{j+1}\phi^{(1)}_i]$, and the matching bin pattern implies $\phi^{(2)}_i>2^j\phi^{(1)}_i$, hence
\[
\theta_i \le 2^{j+1}\phi^{(1)}_i \le 2\,\phi^{(2)}_i
\le 4\frac{\phi^{(1)}_i+\phi^{(2)}_i}{2}.
\]
Therefore $A$ is a $4$-ratio cover of $K$, and
$|A|= |A_2|\le (\lceil\log_2 d\rceil+1)^d$.
We note that the idea of discretizing by dyadic scales and selecting one representative per pattern also appears in \cite{papadimitriou2000approximability}.

\textbf{Idea of the full theorem.}
The warm-ups still fall short: we need a constant ratio $C_1$ and cover size $C_2^d$.
After discretizing $K$ as above, the key additional observation is that \emph{averaging} a finite set of discretized representatives typically provides slack on many coordinates: for a large fraction of discretized points, the average already constant-covers at least half of the coordinates.
This enables the following high-level strategy:
\begin{enumerate}
    \item Discretize points of $K$ into finitely many profiles (as in Warm-up~2), and pick one representative per profile.
    \item Iteratively take averages of the remaining representatives; at each step, keep the average and discard all representatives that are already constant-covered by it on at least half of the coordinates.
    \item For the uncovered coordinates (at most half), apply the induction hypothesis to a suitable lower-dimensional projection.
    \item Combine the averaging-based candidates with the inductive covers so that every coordinate is covered with an overall constant ratio.
\end{enumerate}

\subsection{Constructive proof of Theorem~\ref{thm:ratio-cover}}\label{subsection:ratio-proof}
  We proceed by induction on $d \geq 1$.

  \paragraph{Base case $d=1$.}
  This case follows simply. Let $\theta^{(0)} = \argmax_{\theta \in K} \theta$ (which exists by compactness of $K$).
  Then, $A = \{ \theta^{(0)} \} \subset K$ is a $1$-ratio cover of $K$ with one element.

  \paragraph{Induction.}
  Otherwise, let $d \geq 2$; we assume that the property holds for any dimension $d' < d$, and proceed to show it in dimension $d$.

  We will define below four finite subsets $A_1, \dots, A_4 \subset K$ such that 
  \[A \subset (A_1 + \dots + A_4)/4 = \{ (\phi^{(1)}+ \dots + \phi^{(4)})/4 : \phi^{(1)} \in A_1, \dots, \phi^{(4)} \in A_4 \}\]
  is a $32$-ratio cover of $K$.

  \emph{\underline{Definition of $A_1$}.}
  $A_1$ is constructed exactly as in the first warm-up. For each $j=1, \dots, d$, define $m_j = \sup_{\theta = (\theta_1, \dots, \theta_d) \in K} \theta_j$, which is finite and attained by compactness of $K$.
  Then, for each $j=1, \dots, d$, pick $\ol \theta^{(j)} = (\ol \theta_1^{(j)},\dots,\ol \theta_d^{(j)}) \in K$ such that $\ol \theta_{j}^{(j)} = m_j$, and define $\phi^{(1)} = (\ol\theta^{(1)} + \dots + \ol \theta^{(d)})/d \in K$.
  We then simply take $A_1 = \{ \phi^{(1)} \}$.

  For every $\theta = (\theta_j)_{1 \leq j \leq d} \in K$, define 
  \[J_1 (\theta) = \{ 1 \leq j \leq d : \theta_j \leq m_j/d \}.\]
  Then, by definition of $\phi^{(1)}$, for each $j \in J_1 (\theta)$ one has $\theta_j \leq m_j/d \leq \phi_{j}^{(1)}$,
  i.e., the $j$-th coordinate of $\theta$ is $1$-covered by that of $\phi^{(1)}$.

  \emph{\underline{Definition of $A_2$}.}
  We now proceed to cover some of the remaining coordinates.
  For every subset $J \subseteq [d]$, define 
  \[K_1 (J) = \set{\theta \in K : [d] \setminus J_1 (\theta) = J},\]
  the set of vectors in $K$ whose set of ``remaining coordinates'' is $J$.
  We will now define a subset $A_2 (J) \subset K$ that covers at least half of the coordinates for every vector $\theta \in K_1 (J)$.
  
  If $J = \emptyset$ (meaning that there are no coordinates left to cover), then arbitrarily set $A_2 (J) = \set{\phi^{(1)}}$;
  if $K_1 (J) = \emptyset$ (meaning that there are no vectors to cover), then set $A_2 (J) = \emptyset$.
  Now, let $J \neq \emptyset$ such that $K_1 (J) \neq \emptyset$.
  By definition of $K_1 (J)$, for every $\theta \in K_1 (J)$ and $j \in J$, we know
  $m_j/d < \theta_j \leq m_j$.

  We first discretize the vectors $\set{(\theta_j)_{j \in J} : \theta \in K_1 (J)}$ by powers of $2$ of their coordinates.
  Let us formalize this: we define the \emph{profile} of a vector $\theta \in K_1 (J)$ as the vector 
  \[(k_j)_{j \in J} \in \N^{J}\textrm{ such that }m_j 2^{-k_j - 1} < \theta_j \leq m_j 2^{-k_j}\textrm{ for every }j \in J.\]
  Since $m_j/d < \theta_j \leq m_j$, one has $0 \leq k_j \leq \left\lceil \log_2 d \right\rceil - 1 \leq \log_2 d$.
  
  We say that a profile
  is admissible if an element of $K_1 (J)$ has this profile.
  We then define a subset $S \subset K_1 (J)$ containing exactly one representative of each admissible profile.
  Note that the cardinality of $S$ is at most $|S| \leq | \set{0, \dots, \left\lceil \log_2 d \right\rceil - 1}^{J}| \leq (\log_2 d + 1)^d$.

  So far, we have produced a finite set of representatives that represent the (possibly infinite) remaining values of $\theta$. Since each $\theta$ has its coordinates approximated by its representative, our intuition is that producing a ratio-cover for this set of representatives will yield a ratio-cover for all $\theta$.
  In this sense, we have essentially reduced the task of ratio-covering a convex set $K$, to the task of providing a ratio cover for a set with bounded size. At first, this may not seem much easier; particularly because the set is still quite large. However, let us first consider a weaker goal, where we only ratio-cover half of the coordinates:
  \begin{definition}[Half-$c$-cover]
      For $c \geq 1$, we say that a vector $\theta$ is half-$c$-covered by a vector $\mu$ if the inequality $\theta_j \leq c \mu_j$ holds for at least half of the indices $j \in J$.
  \end{definition}

  Even producing a half-$c$-cover for our bounded-size set is not immediately obvious. However, we will argue that \textit{the average of the set $S$} is a half-$4$-cover for at least half of the elements in $S$:

  \begin{claim}
    \label{cla:half-cover-mean}
    Given a non-empty set of indices $J \subseteq [d]$, let $S' \subset \R_+^d$ be a finite set, and define $\mu' = \frac{1}{|S'|} \sum_{\theta' \in S'} \theta'$.
    Then, at least half of the elements of $S'$ are half-$4$-covered by $\mu'$.
  \end{claim}

  \begin{proof}[Proof of Claim~\ref{cla:half-cover-mean}]
    Denote by $B \subset S'$ the set of $\theta' \in S'$ that are not half-$4$-covered by $\mu'$.
    For each $j \in J$, Markov's inequality implies that $|\set{\theta \in S' : \theta_j > 4 \mu'_j}| \leq |S'|/4$.
    Hence,
    \begin{equation*}
      \frac{|S'|}{4}
      \geq \frac{1}{|J|} \sum_{j \in J} \sum_{\theta \in S'} \ind{\theta_j > 4 \mu'_j}
      = \sum_{\theta \in S'} \frac{\abs{\set{j \in J : \theta_j > 4 \mu_j'}}}{|J|}
      \geq |B| \cdot \frac{|J|/2}{|J|}
      = \frac{|B|}{2}
      \, ,
    \end{equation*}
    so that $|B| \leq |S'|/2$, which concludes the proof.
  \end{proof}
  Our plan is to invoke this claim repeatedly to produce a half-$c$-cover for all of $S$. Concretely, we recursively define elements $\mu^{(k)} \in K$ and subsets $S_k \subset S$ as follows. Let $S_1 = S$. Then, for each $k \geq 1$:
  \begin{itemize}
  \item Define $\mu^{(k)} = \frac{1}{|S_k|} \sum_{\theta \in S_k} \theta \in K$.
  \item Let $S_{k+1}$ be the set of vectors in $S_k$ that are not half-$4$-covered by $\mu^{(k)}$.
  \item If $S_{k+1} = \emptyset$, then stop; otherwise, turn to step $k+1$.
  \end{itemize}
  By Claim~\ref{cla:half-cover-mean}, one has $|S_{k+1}| \leq |S_k|/2$, hence the loop stops after
  \begin{equation*}
    r \leq \left\lfloor \log_2 |S| \right\rfloor + 1 \leq d \log_2 (\log_2 d + 1) + 1
    \leq 2^d
  \end{equation*}
  iterations.
  We define $A_2 (J) = \{ \mu^{(1)}, \dots, \mu^{(r)} \}$, the set of partial means, such that $|A_2(J)| \leq 2^d$ by the above.
  
  By construction of the means $\mu^{(k)}$ and sets $S_k$ above, each $\theta \in S$ is half-$4$-covered by an element of $A_2 (J)$.
  In addition, by definition of $S$, each vector $\theta \in K_1 (J)$ has all its coordinates $j \in J$ within a factor of $2$ of those of an element in $S$. Hence, $\theta$ is also half-$8$-covered by an element $\phi^{(2)} = \phi^{(2)} (\theta) \in A_2 (J)$.
  We denote the coordinates handled by this step by
  \[J_2 (\theta) = \{ j \in J : \theta_j \leq 8 \phi_j^{(2)} \}.\]
  Since this covers at least half of the coordinates for any $\theta$, we know $|J_2 (\theta)| \geq |J|/2$. Thus, the number of remaining unhandled coordinates is at most
  \begin{equation*}
      |[d] \backslash (J_1(\theta) \cup J_2(\theta)) | \le d - |J_1(\theta)| - |[d] \backslash J_1(\theta)|/2 \le d/2.
  \end{equation*}

  In summary, define $A_2 = \bigcup_{J \subseteq [d]} A_2 (J)$, such that $|A_2| \leq \sum_{J \subseteq [d]} 2^d = 2^{2d}$.
  Since $K = \bigcup_{J \subseteq [d]} K_1 (J)$, it follows from the above that for every $\theta \in K$, there exists $\phi^{(2)} = \phi^{(2)} (\theta) \in A_2$ such that all coordinates
  $j \in J_2 (\theta)$ of $\theta$ are $8$-covered by the corresponding ones of $\phi^{(2)}$.

  \emph{\underline{Definition of $A_3$}.}
  For any $\theta \in K$, let $J'(\theta)$ denote the coordinates not handled by $A_1$ and $A_2$:
  \[ J'(\theta) = [d] \backslash (J_1(\theta) \cup J_2(\theta)).\]
  As noted before, the number of unhandled coordinates is at most $d/2$.
  In what follows, we will handle the remaining coordinates $J'(\theta)$ by induction, together with an argument to avoid degrading the ratio factor.
  
  For every subset $J \subset [d]$ with $|J| \leq d/2$, we define $\pi_J(K)$ as the projection of $K$ to the coordinates $J$, meaning
  \[ \pi_J(K) = \{ (\theta_j)_{j \in J} : (\theta_j)_{1 \le j \le d} \in K \}.\]
  
  The set $\pi_J(K)$ is a convex compact subset in dimension $d' \leq d/2 < d$; hence, by the induction hypothesis, it admits a subset with at most $2^{8 d'} \leq 2^{4d}$ elements that is a $32$-ratio cover.
  
  We define the corresponding subset $A_3 (J) \subset K$, with at most $2^{4d}$ elements, where $\pi_J (A_3 (J))$ is a 32-ratio cover of $\pi_J (K)$. For $J = \emptyset$, we set $A_3 (\emptyset) = \{\phi^{(1)}\}$.
  By the ratio cover property, for every $\theta \in K$, there exists $\phi^{(3)} = \phi^{(3)} (\theta) \in A_3 (J'(\theta))$ such that $\theta_j \leq 32 \, \phi_{j}^{(3)}$ for every $j \in J'(\theta)$.
  From here, we define $A_3 = \bigcup_{J \pp |J| \leq d/2} A_3 (J)$, with cardinality $|A_3| \leq \sum_{J \pp |J| \leq d/2} |A_3 (J)| \leq 2^d \times 2^{4 d} = 2^{5 d}$.

  At this point, it is worth noting that achieving a $32$-covering from $\phi^{(3)}$ over the coordinates $j \in J'(\theta)$ is not sufficient for our purposes, since the vector $\phi^{(3)}$ is to be averaged with previous vectors $\phi^{(1)}, \phi^{(2)}$ to cover coordinates $j \in [d] \setminus J'(\theta)$.
  This averaging operation will degrade the constant $32$ in the ratio factor, thus breaking the induction.

  In order to address this issue, we distinguish between two types of coordinates.
  For $\theta \in K$, let $\phi^{(3)} = \phi^{(3)} (\theta) \in A_3 (J'(\theta))$ be as above.
  We define 
  \[J_3 (\theta) = \set{j \in J'(\theta) : \theta_j \leq 8 \,\phi_{j}^{(3)}}\textrm{ and }J_4 (\theta) = J'(\theta) \setminus J_3 (\theta) = [d] \setminus (J_1 (\theta) \cup J_2 (\theta) \cup J_3 (\theta)).\]
  By definition, coordinates $j \in J_3 (\theta)$ are $8$-covered, which provides the necessary slack to maintain a $32$-ratio covering after the averaging operation.
  It thus remains to handle the coordinates in $J_4 (\theta)$.

  \emph{\underline{Definition of $\phi^{(4)}$}.} We will not actually define a set $A_4$ as our earlier definition suggests, but instead will define a function $\phi^{(4)}(\theta)$ as before.
  
  The final idea of the proof is the following: coordinates $j \in J_4 (\theta)$ are $32$-covered but not $8$-covered by $\phi^{(3)} (\theta)$; meaning, the value of $\theta_j$ is known up to a factor of $4$, and can thus be separately encoded. Specifically, for any $\phi^{(3)} \in A_3$ and subset $J \subset [d]$, let
  \[C (\phi^{(3)}, J) = \set{\theta \in K : 8 \phi_{j}^{(3)} < \theta_j \leq 32 \phi_{j}^{(3)},\ \textrm{for all}\ j \in J}.\]
  
  Using this, let $R(\phi^{(3)}, J) \in C (\phi^{(3)}, J)$ be an arbitrary representative chosen for $C (\phi^{(3)}, J)$ (as long as the set is nonempty).   For any $\theta \in K$, let $\phi^{(3)} = \phi^{(3)} (\theta)$ and $J_4 = J_4 (\theta)$ as above. 
  By definition of $J_4 (\theta)$, we know $\theta \in C (\phi^{(3)}, J_4)$; hence, the latter set is nonempty, and we choose $\phi^{(4)}(\theta) = R (\phi^{(3)}, J_4)$.
  This ensures that $\theta_j \leq 32 \phi_{j}^{(3)} \leq 4 \phi_{j}^{(4)}$ for every $j \in J_4$.

  \emph{\underline{Conclusion}.}
  We let
  \begin{equation*}
    A = \bigg\{ \frac{\phi^{(1)} + \phi^{(2)} (\theta) + \phi^{(3)} (\theta) + \phi^{(4)} (\theta)}{4} : \theta \in K \bigg\}
    \, ,
  \end{equation*}
  which is contained in $K$ by convexity.
  In addition, for every $\theta \in K$, letting $\phi = (\phi^{(1)} + \phi^{(2)} (\theta) + \phi^{(3)} (\theta) + \phi^{(4)} (\theta))/4 = (\phi_j)_{1 \leq j \leq d}$ and $J_k = J_k (\theta)$ for $k=1, \dots, 4$, we have by the above that:
  \begin{equation*}
    \theta_j \leq 4 \phi_j \text{ for } j \in J_1 ; \quad
    \theta_j \leq 32 \phi_j \text{ for } j \in J_2; \quad
    \theta_j \leq 32 \phi_j \text{ for } j \in J_3; \quad
    \theta_j \leq 16 \phi_j \text{ for } j \in J_4 .
  \end{equation*}
  Since $[d] = J_1 \cup J_2 \cup J_3 \cup J_4$, we deduce that $A$ is a $32$-ratio cover of $K$.

  It remains to control the cardinality of $A$.
  There are at most $|A_2| \leq 2^{2d}$ possible values of $\phi^{(2)} (\theta) \in A_2$, and at most $|A_3| \leq 2^{5d}$ possible values for $\phi^{(3)} (\theta) \in A_3$; in addition, since $\phi^{(4)} (\theta) = R(\phi^{(3)} (\theta),J)$ for some $J \subset [d]$, there are at most $2^d$ possible values of $\phi^{(4)} (\theta)$ given $\phi^{(3)} (\theta)$.
  Hence, there are at most $2^{2d} \times 2^{5d} \times 2^d \leq 2^{8 d}$ possible values for $(\phi^{(2)} (\theta), \phi^{(3)} (\theta), \phi^{(4)} (\theta))$ for $\theta \in K$, and therefore $|A| \leq 2^{8d}$. \qed

\subsection{Approximate Pareto curves}\label{subsection:apply-pareto}
Consider a setting where a set $K \subset \R_+^d$ represents the attainable objective values in a multi-objective optimization problem.
For instance, one may think of producing $d$ different products: a point $x=(x_1,\dots,x_d)\in K$ indicates that it is possible to simultaneously produce $x_j$ units of product~$j$ for each $j\in[d]$.
Such problems typically involve tradeoffs between objectives, and a common object of interest is the \emph{Pareto curve} (or Pareto frontier), namely the set of points in $K$ that are not dominated coordinatewise by any other point in $K$. In general, the Pareto curve may be infinite.
Papadimitriou and Yannakakis~\cite{papadimitriou2000approximability} showed that one can nevertheless extract a finite \emph{$\eps$-approximate Pareto curve}, meaning a subset $A\subset K$ such that for every $\theta=(\theta_1,\dots,\theta_d)\in K$ there exists $\phi=(\phi_1,\dots,\phi_d)\in A$ with
\[
\theta_j \le (1+\eps)\,\phi_j\quad\text{for all }j\in[d].
\]
In our terminology, this is exactly a $(1+\eps)$-ratio cover of $K$.
Their result does not assume convexity of $K$, but it does assume that each coordinate ranges over a bounded interval.

\begin{theorem}[Rephrased Theorem~1 of Papadimitriou and Yannakakis~\cite{papadimitriou2000approximability}]
\label{thm:py20}
For every set $K \subset [\frac{1}{R},R]^d$, every $0 < \eps \le 1$, and every $R \ge 1$, there exists a subset $A\subset K$ with
\[
|A| \le \max\left\{1, \left(c\,\frac{\log R}{\eps}\right)^{d-1}\right\}
\]
that is an $\eps$-approximate Pareto curve (equivalently, a $(1+\eps)$-ratio cover) for $K$, where $c>0$ is a universal constant.
\end{theorem}

Their short proof is essentially the same dyadic discretization idea that we used in our second warm-up.
A key feature of \Cref{thm:py20} is its dependence on the range parameter $R$.
If one adds the assumption that $K$ is convex, then the averaging idea from our first warm-up yields a simple way to remove this dependence.

\begin{corollary}[Corollary of \Cref{thm:py20} under convexity]\label{cor:py-convex}
For every convex and compact set $K \subset \R_+^d$ and every $0 < \eps \le 1$, there exists a subset $A\subset K$ with
\[
|A| \le \left(c\,\frac{\log d}{\eps}\right)^{d-1}
\]
that is an $\eps$-approximate Pareto curve for $K$, where $c>0$ is a universal constant.
\end{corollary}

By comparison, we now use our ratio cover theorem to obtain an $\eps$-approximate Pareto curve whose size has no dependence on $R$ and whose base does not involve $d$.
The convexity assumption is essential for this stronger statement (see the discussion following \Cref{thm:ratio-cover}), and it is natural in many applications, for example when $K$ is the feasible region of a linear program.

\begin{corollary}[of \Cref{thm:ratio-cover}]
  \label{cor:eps-pareto}
  For every convex and compact set $K \subset \R_+^d$ and every $0 < \eps \le 1$, there exists an $\eps$-approximate Pareto curve $A \subset K$ with
  \[
  |A| \le 2^{8d}\left(3 + \frac{4\log_2(128/\eps)}{\eps}\right)^{d-1}.
  \]
\end{corollary}

The proof of \Cref{cor:eps-pareto} is deferred to \Cref{sec:remainingproofs}.
A substantial literature following Papadimitriou and Yannakakis develops algorithms that construct $\eps$-approximate Pareto curves of near-minimal size (see, e.g., \cite{vassilvitskii2005efficiently,diakonikolas2008succinct,diakonikolas2010small,bazgan2015approximate,daskalakis2016good}).
Our result shows that under convexity, there always exists an $\eps$-approximate Pareto curve of smaller size, and therefore yields improved existential upper bounds for what such approximation algorithms can aim to achieve in this setting.

\subsection[Proof of a version of Theorem~\ref{thm:ratio-cover} via epsilon-nets]%
{Proof of a version of \Cref{thm:ratio-cover} based on the reduction to $\varepsilon$-nets}

\label{subsection:epsilonnetproof}

In this section, we provide an alternative proof of \Cref{thm:ratio-cover} based on $\varepsilon$-net arguments in the spirit of Haussler and Welzl~\cite{haussler1986epsilon}, introduced in the context of Pareto curves by Diakonikolas and Yannakakis~\cite{diakonikolas2010small} for general sets $K$.
We then combine this approach with convexity via a volumetric argument similar to the one used in the proof of \Cref{thm:local-entropy-mixtures}. We state the resulting bound as a separate theorem because it yields a different dependence on $\varepsilon$ in the language of \Cref{cor:eps-pareto}.

\begin{theorem}
\label{thm:secratiotheorem}
For every convex and compact set $K \subset \R_+^d$ with nonempty interior and every $ \varepsilon > 0$, there exists an $\varepsilon$-approximate Pareto curve $A \subset K$ with
\[
|A| \le \left\lceil 8d\left(1 + \frac{1}{\varepsilon}\right)^d \log\left(13\left(1 + \frac{1}{\varepsilon}\right)^d\right)\right\rceil.
\]
\end{theorem}

Note that, in full generality, the dependence on $\varepsilon$ in \Cref{thm:secratiotheorem} differs from that in \Cref{cor:eps-pareto}. For instance, in the regime of fixed $d$ and sufficiently small $\varepsilon$, the bound of \Cref{cor:eps-pareto} is stronger because it has exponent $d-1$ rather than $d$. Despite not having the same dependence on $\varepsilon$ as \Cref{cor:eps-pareto}, the result of \Cref{thm:secratiotheorem} implies a $32$-ratio cover of size at most
\[
\left\lceil 8d\left(\frac{32}{31}\right)^d \log\left(13\left(\frac{32}{31}\right)^d\right)\right\rceil \le 2^{5d},
\]
for all $d \ge 1$.
Since optimization of numerical constants is not our priority, we mostly use the simple integer form of \Cref{thm:ratio-cover}.

We also note that the construction in our first proof of \Cref{cor:eps-pareto} is deterministic and constructive, whereas the proof of \Cref{thm:secratiotheorem}
is nonconstructive and based on a probabilistic argument.

\begin{proof}
The nonempty interior assumption implies $\vol(K)>0$. For each $\theta \in K$, set
\[
R_{\theta}(\varepsilon) = \{\phi \in K: \theta_j\le (1 + \varepsilon)\phi_j\ \textrm{for all}\ j \in [d]\}.
\]
As observed in \cite[Section~4.1]{diakonikolas2010small}, $A$ is an $\varepsilon$-approximate Pareto curve if and only if, for all $\theta \in K$,
\[
R_\theta(\varepsilon) \cap A \neq \emptyset.
\]
Indeed, this means that for any $\theta\in K$ there exists $\phi\in A$ such that $\theta_j\le (1 + \varepsilon)\phi_j$ for all $j \in [d]$.

Now define the collection of sets $\mathcal{R}(\varepsilon) = \{R_{\theta}(\varepsilon): \theta \in K\}$. Following \cite{haussler1986epsilon}, the VC dimension of the \emph{range space} $(K,\mathcal{R}(\varepsilon))$ is the largest integer $v$ such that $\mathcal{R}(\varepsilon)$ shatters some set $X\subset K$ of size $v$, where shattering means that for every subset $Y\subseteq X$ there exists a range $R\in\mathcal{R}(\varepsilon)$ such that $R\cap X = Y$. The following result is proved in \cite{diakonikolas2010small}:
\begin{lemma}[Lemma~4.5 in \cite{diakonikolas2010small}]
\label{lem:diakonikolassmall}
For any $\varepsilon>0$, the VC dimension of the range space $(K,\mathcal{R}(\varepsilon))$ is at most $d$.
\end{lemma}

Note that \Cref{lem:diakonikolassmall} does not require $K$ to be convex. We now depart from the reduction in \cite{diakonikolas2010small} and exploit convexity. For each $\theta \in K$ set
\[
Q_\theta(\varepsilon) \triangleq \frac{1}{1+\varepsilon}\theta + \frac{\varepsilon}{1+\varepsilon}K.
\]
By convexity of $K$, we have $Q_\theta(\varepsilon) \subset K$ for every $\theta \in K$. Moreover, for every $\theta \in K$ it holds that
\[
Q_\theta(\varepsilon) \subset R_\theta(\varepsilon).
\]
Indeed, any $\phi \in Q_\theta(\varepsilon)$ can be written as $\phi = \frac{1}{1+\varepsilon}\theta + \frac{\varepsilon}{1+\varepsilon}\nu$ with $\nu \in K$, and then for all $j \in [d]$,
\[
\theta_j \le (1 + \varepsilon)\left(\frac{1}{1+\varepsilon}\theta_j + \frac{\varepsilon}{1+\varepsilon}\nu_j\right) = (1 + \varepsilon)\phi_j,
\]
so $\phi \in R_\theta(\varepsilon)$. Let $\mu$ be the uniform probability measure on $K$, which is well-defined since $\vol(K)>0$. Then for any $\theta \in K$, a volumetric argument gives
\[
\mu(R_\theta(\varepsilon)) \ge \mu(Q_\theta(\varepsilon)) = \left(\frac{\varepsilon}{1 + \varepsilon}\right)^d.
\]
Thus, it suffices to construct a set $A$ such that for all $\theta \in K$ one has $A \cap R_\theta(\varepsilon) \neq \emptyset$.
Now the class $\mathcal{R} (\eps)$ has VC dimension at most $d$, and
every range satisfies the uniform lower bound $\mu(R_\theta(\varepsilon)) \ge \eta$, where
$
\eta = \left(\frac{\varepsilon}{1 + \varepsilon}\right)^d.
$
Thus % by the standard $\varepsilon$-net argument
by a standard VC lower bound
(see \cite{haussler1986epsilon} and \cite[Theorem~2.1]{blumer1989learnability}), if we sample
\[
\left\lceil\frac{8d}{\eta}\log\left(\frac{13}{\eta}\right)\right\rceil
\]
points independently according to $\mu$ and let $A$ be the resulting set, then with probability at least $1/2$ we have $A \cap R_\theta(\varepsilon) \neq \emptyset$ for all $\theta \in K$. Plugging in the value of $\eta$ yields the stated bound. Existence follows since this event has positive probability.
\end{proof}

\section{Mixture models (convex aggregation)}\label{sec:mixture-models}
We are ready to state our second main statistical result.
\begin{theorem}
\label{thm:convexaggregation}
    Let $X_{1}, \ldots, X_n$ with $n \ge 2$ be an i.i.d.\ sample from some unknown density $p^\star$ that belongs to the convex hull of a known class of arbitrary densities $\{p_1, \ldots, p_\M\}$, namely $\mathcal{P} = \operatorname{conv}\{p_1, \ldots, p_\M\}$ defined on the space $\X$ with respect to the measure $\mu$ and $p^\star \in \mathcal P$. There is a density estimator $\widehat{p}$ such that, with probability at least $1 - \delta$,
    \[
    \kl\left(p^\star, \widehat{p}\right) \le \frac{870030\,M + 45000\log(M)\log\left(4/\delta\right)}{n}.
    \]
In the regime $n \ge \log(1/\delta)$, provided that $n \ge M \ge 5000$, this bound is optimal in the sense that there are $\M$ distributions such that no estimator can perform better up to absolute constant factors, as shown in \cite[Theorem 4]{mourtada2025estimation}. 
\end{theorem}
As above, for the sake of presentation we assume that we are given a sample of size $2n$ (we translate back to the original sample size at the end). In general, provided that the sample size is at least $2$ we can split the sample into two equally sized independent parts, while dropping at most one element of the original sample. This will only lead to a larger absolute constant in the final statement. We describe the estimator $\widehat{p}$ of Theorem \ref{thm:convexaggregation} as follows:
\begin{framed}
\begin{itemize}
\item %\jao{change}
  Use the first $n$ observations to run a Birg\'{e}-Le Cam tournament (see \cref{lem:lecam-birge-local}) over $\mathcal{P}$ to find a convex closed set $\wh \Qs_1 \subset \mathcal{P}$ such that, with probability $1 - \delta/2$, it satisfies
\begin{equation}
\label{eq:optestimationinhellinger}
\sup\limits_{p, q \in \wh \Qs_1}\hels(p, q) \le \frac{3528\log(64)\,M + 512\log(4/\delta)}{n}, \quad\textrm{and}\quad p^\star \in \wh \Qs_1.
\end{equation}
\item Let $\Delta$ be the simplex in $\R^M$. Consider the closed convex set $\wh \Delta \subseteq \Delta$ defined as follows:
\[
\wh \Delta = \left\{\theta \in \Delta: \sum_{j = 1}^M \theta_{j}p_j \in \wh \Qs_1\right\},
\]
and construct a $32$-ratio cover of $\wh \Delta$ with cardinality at most $2^{8 M}$, denoted as $\wh \Delta_c$.
\item Choose any $\beta^{(1)}, \ldots, \beta^{(M)} \in \wh \Delta$ such that for all $j \in [M]$, $\beta^{(j)}_j = \sup\{\theta_j: \theta \in \wh \Delta\}$, i.e., each vector $\beta^{(j)}$ maximizes the $j$-th coordinate in $\wh \Delta$. Construct a (random) density class 
\begin{equation}
\label{eq:wprimesetnew}
\wh \Qs_2 = \left\{\frac{1}{2}\sum\limits_{j = 1}^M \theta_j p_j + \frac{1}{2M}\sum\limits_{j = 1}^M\sum\limits_{k = 1}^M\beta^{(j)}_kp_k: \theta \in \wh \Delta_c \right\}.
\end{equation}
\item Using the second half of the sample, run the density estimator of \Cref{lem:onlinetobatch} over $\wh \Qs_2$ to output the final density $\widehat{p}$. 
\end{itemize}
\end{framed}

\paragraph{Proof of Theorem \ref{thm:convexaggregation}.}
Fix $\delta \in (0, 1)$.
Combining the Birg\'{e}-Le Cam tournament result (\cref{lem:lecam-birge-local}) and the local Hellinger entropy bound of \cref{thm:local-entropy-mixtures}, we construct the desired convex set $\wh \Qs_1 \subset \mathcal P$ satisfying, with probability at least $1 - \delta/2$,
\begin{equation}
\label{eq:diametersquaredbound}
\sup\limits_{p, q \in \wh \Qs_1}\hels(p, q) \le \frac{3528\log(64)\,M + 512\log(4/\delta)}{n}, \quad\textrm{and}\quad p^\star \in \wh \Qs_1 .
\end{equation}
From now on we work under this event.

Using the second half of the sample, we output the density $\widehat{p}$ obtained using the procedure of \Cref{lem:onlinetobatch} with respect to $\wh \Qs_2$ defined by \eqref{eq:wprimesetnew}.
Let $\theta \in \wh \Delta$ be such that $p^\star = \sum\limits_{j = 1}^M \theta_j p_j$.
By definition of $\wh \Delta_c$, there exists a vector $\phi \in \wh \Delta_c$ such that
\begin{equation}
  \label{eq:ratiolabel}
  \theta_j \leq 32 \, \phi_j
  \quad \text{for every } j=1, \dots, M.
\end{equation}
We then define
\begin{equation*}
  p^\star_{c}
  = \sum_{j=1}^M \phi_j p_j
  \quad \text{and} \quad
  \widetilde{p} = \frac{1}{2}p^\star_c + \frac{1}{2M}\sum\limits_{j = 1}^M\sum\limits_{k = 1}^M\beta^{(j)}_kp_k.
\end{equation*}
Note that $\widetilde{p} \in \wh \Qs_2$ by \eqref{eq:wprimesetnew}. We aim to use  \Cref{lem:onlinetobatch} to bound $\kll{p^\star}{\widehat{p}}$, so we will first verify its condition on the densities. 
Writing any $p\in \wh \Qs_2$ as $p_\alpha=\frac12\sum_k(\alpha_k+\bar\beta_k)p_k$ with $\bar\beta_k=\tfrac1M\sum_j\beta^{(j)}_k$ and noting $\alpha_k\le \sum_j\beta^{(j)}_k=M\bar\beta_k$, we get
\begin{equation*}
\sup_{p_\alpha, p_{\alpha'} \in \wh \Qs_2} \sup_{x \in \mathcal{X}} \log\left(\frac{p_\alpha(x)}{p_{\alpha'}(x)}\right) \le \log\left(1+\frac{\sum_k \alpha_k p_k(x)}{\sum_k \bar\beta_k p_k(x)} \right)\le \log \left(1+M\right) \le \log(2M).
\end{equation*}

Therefore, by \Cref{lem:onlinetobatch}, and since by Theorem \ref{thm:ratio-cover} it holds that $|\wh \Qs_2| \le 256^M$, we have, with probability at least $1 - \delta/2$,
\begin{equation}\label{eq:conv-agg-form}
\kll{p^\star}{\widehat{p}} \le \kll{p^\star}{\widetilde{p}} + \frac{2M\log(256) + \frac{25}{4}\left(e-2\right)\left(\log(2M)+1\right)\log\left(\frac{4}{\delta}\right)}{n}.
\end{equation}
Finally, by Lemma \ref{lem:kltohellinger} and the convexity of the Hellinger distance squared, and observing that the ratio of $p^\star$ to $\frac{1}{2}p^\star_c + \frac{1}{2M}\sum\limits_{j = 1}^M\sum\limits_{k = 1}^M\beta^{(j)}_kp_k$ is bounded by $64$ by \eqref{eq:ratiolabel}, and, in particular,
$
p^\star(x)\le 32\,p^\star_c(x).
$
Therefore,
\begin{align*}
\kl(p^\star, \widetilde{p}) &\le \frac{2\max\{1, \log(64)\}}{(\sqrt{e}-1)^2}\hels\left(p^\star, \frac{1}{2}p^\star_c + \frac{1}{2M}\sum\limits_{j = 1}^M\sum\limits_{k = 1}^M\beta^{(j)}_kp_k\right)
\\
&\le \frac{2\log(64)\left(3528\log(64)\,M + 512\log(4/\delta)\right)}{(\sqrt{e}-1)^2\,n},
\end{align*}
where in the last line we used that $p^\star \in \wh \Qs_1$ and $\frac{1}{2}p^\star_c + \frac{1}{2M}\sum\limits_{j = 1}^M\sum\limits_{k = 1}^M\beta^{(j)}_kp_k \in \wh \Qs_1$ combined with \eqref{eq:diametersquaredbound}. Plugging this into \eqref{eq:conv-agg-form} and using the union bound, we have with probability at least $1 - \delta$,
\begin{align*}
\kl(p^\star, \widehat{p}) &\le \frac{2M\log(256) + \frac{25}{4}\left(e-2\right)\left(\log(2M)+1\right)\log\left(\frac{4}{\delta}\right)}{n} 
\\
&\qquad+\frac{2\log(64)\left(3528\log(64)\,M + 512\log\left(4/\delta\right)\right)}{(\sqrt{e}-1)^2n}.
\end{align*}
Note that we used a sample of size $2n$, which leads to an increased absolute constant by at most a factor of $3$. 
The claim follows.\qed

\section{Applications of ratio covers}\label{sec:apps}
In this section, we will discuss two more applications of our technique for ratio covering of convex sets developed in Section~\ref{sec:domin-subs-conv}. 
In Subsection~\ref{sec:yangbarron-sharp}, we provide a sharpening (in some cases) of the result of Yang-Barron for obtaining expected error guarantees from KL covers, which will also yield our desired expected error guarantee for mixture models.  
In Subsection~\ref{subsection:apply-kl-cover}, we partially resolve a problem of Tang~\cite{tang2022divergence} about KL covers for the simplex via Theorem~\ref{thm:ratio-cover}.

\subsection{Expected error guarantees: sharpening the Yang--Barron bound}
\label{sec:yangbarron-sharp}

The celebrated result of Yang and Barron~\cite{yang1999information} shows how to turn a finite KL cover into an \emph{expected} KL risk bound via a progressive mixture construction. In this section, we explain how our ratio-cover approach yields a sharper guarantee when the underlying density class is convex and we may construct appropriate ratio covers. We use the following form of the Yang--Barron bound.

\begin{theorem}[Yang--Barron; see the proof of Theorem~32.1 in Section~32.2.1 of~\cite{polyanskiy2025information}]
\label{thm:yang-barron}
Let $X_{1}, \ldots, X_n$ be \iid from a density $p^\star$ on $(\X,\mu)$, and let $q_1,\dots,q_N$ be densities such that
\[
\min_{j\in[N]}\kl(p^\star,q_j)\le \eps.
\]
Define the progressive mixture estimator $\widehat p \triangleq \sum_{j=1}^N w_j q_j$, where
\[
w_j \triangleq \frac{1}{n+1}\sum_{i=1}^{n+1}
\frac{\prod_{t=1}^{i-1} q_j(X_t)}{\sum_{k=1}^N \prod_{t=1}^{i-1} q_k(X_t)}.
\]
Then
\[
\E\big[\kl(p^\star,\widehat p)\big]\le \eps + \frac{\log N}{n+1}.
\]
\end{theorem}

Two issues prevent a direct use of \Cref{thm:yang-barron} in our setting. First, constructing KL covers is substantially more delicate than constructing Hellinger covers. Second, even when one has sharp KL covering numbers (e.g., for the simplex), applying Yang--Barron with a \emph{global} cover introduces unnecessary logarithmic factors, since it does not exploit localization.

Our approach is to localize the construction. We first build an estimator that is close to $p^\star$ in Hellinger distance using the Le Cam--Birg\'e tournament (\Cref{lem:lecam-birge-local}). We then construct a finite candidate list that KL-covers the resulting local Hellinger uncertainty region using ratio covers (\Cref{sec:domin-subs-conv}), and finally apply \Cref{thm:yang-barron} to this local list. The remaining technical point is that the Hellinger control is in expectation, so $p^\star$ may fall outside the chosen ball; we address this by showing that the KL quality of the ratio-based cover degrades gracefully with $\hels(p^\star,\widehat p)$ (see \Cref{lemma:kl-from-ball}).

\begin{lemma}\label{lemma:kl-from-ball}
Let $\mathcal P$ be a convex class of densities on $\X$ with respect to $\mu$, fix $p\in\mathcal P$, $\eps>0$, and $\alpha\ge 2$.
Set
\[
\mathcal Q \triangleq \mathcal P \cap B_{\hels}(p,\eps),
\qquad
B_{\hels}(p,\eps)\triangleq\{q\ \text{is a density}:\ \hels(p,q)\le \eps\}.
\]
Assume that $S\subset \mathcal Q$ is finite and has the following property: for every $q\in\mathcal Q$ there exists $\tilde q\in S$ such that
\[
q(x)\le \alpha\,\tilde q(x)\qquad\text{for all }x\in\X.
\]
Then for every $q\in\mathcal P$,
\[
\min_{\tilde q\in S}\kl(q,\tilde q)
\le
30\,\bigl(\eps+\hels(p,q)\bigr)\,
\log\Bigl(\alpha\Bigl(1\vee \frac{\hels(p,q)}{\eps}\Bigr)\Bigr).
\]
\end{lemma}

The proof of this result is deferred to \Cref{sec:remainingproofs}. Combined with the tail bounds on the Hellinger estimator, we obtain an expected KL guarantee for convex classes admitting suitable ratio covers.

\begin{proposition}\label{prop:expected-recipe}
Assume that $n$ is even.
Let $\mathcal{P}$ be a convex class of densities and assume that $X_1,\ldots,X_n$ be i.i.d.\ from some $p^\star\in\mathcal{P}$.
Let $\eps_{n/2}>0$ satisfy
\[
\frac{n}{2}\eps_{n/2}^2 \ge \log \nloc(\mathcal{P},\eps_{n/2})\,\lor\,1.
\]
Assume that there exist $\alpha\ge 2$ and an integer $N_S\ge 1$ such that for every $p\in\mathcal{P}$ one can find a finite set
$S(p)\subset \mathcal{P}\cap B_{\hels}(p,\eps_{n/2}^2)$ which is an $\alpha$-ratio cover of $\mathcal{P}\cap B_{\hels}(p,\eps_{n/2}^2)$ (in the sense of \Cref{lemma:kl-from-ball}) and satisfies $|S(p)|\le N_S$.
Then, there exists an estimator $\widehat{p}$ such that
\[
\E\big[\kl(p^\star,\widehat{p})\big]
\le
366609\,\eps_{n/2}^2 \log(\alpha) + \frac{\log(N_S)}{n/2+1}.
\]
\end{proposition}

\begin{proof}
Split the sample into two independent halves.
Let $\widehat{p}_{\hel}$ be the Le Cam--Birg\'e estimator (\Cref{lem:lecam-birge-local}) built from the first $n/2$ observations, and let $\widehat{p}$ be the Yang--Barron mixture (Theorem~\ref{thm:yang-barron}) built from the remaining $n/2$ observations over the candidate set $S(\widehat{p}_{\hel})$.
Conditioning on the first half and applying Theorem~\ref{thm:yang-barron} to the second half yields
\[
\E\big[\kl(p^\star,\widehat{p})\big]
\le
\E\Big[\min_{q\in S(\widehat{p}_{\hel})}\kl(p^\star,q)\Big]
+\frac{\log(N_S)}{n/2+1}.
\]
By Lemma~\ref{lemma:kl-from-ball} with $p=\widehat{p}_{\hel}$ and $\eps=\eps_{n/2}^2$,
\[
\min_{q\in S(\widehat{p}_{\hel})}\kl(p^\star,q)
\le
30\Bigl(\eps_{n/2}^2+\hels(\widehat{p}_{\hel},p^\star)\Bigr)
\log\Bigl(\alpha\Bigl(1\vee \frac{\hels(\widehat{p}_{\hel},p^\star)}{\eps_{n/2}^2}\Bigr)\Bigr).
\]
For $x\in(0,1)$, \Cref{lem:lecam-birge-local} applied with sample size $n/2$ gives, with probability at least $1-x$,
\[
\hels(\widehat{p}_{\hel},p^\star)
\le
882\,\eps_{n/2}^2+\frac{256\log(2/x)}{n}.
\]
Therefore,
\begin{align*}
\E\big[\kl(p^\star,\widehat{p})\big]
&\le
30 \int_0^1
\Bigl(883\,\eps_{n/2}^2+\frac{256\log(2/x)}{n}\Bigr)
\log\left(\alpha\Bigl(882+\frac{256\log(2/x)}{\eps_{n/2}^2\,n}\Bigr)\right)\,dx
+\frac{\log(N_S)}{n/2+1}.
\end{align*}
Since $(n/2)\eps_{n/2}^2\ge 1$, we have $\eps_{n/2}^2\ge 2/n$, and hence
\[
\frac{256\log(2/x)}{\eps_{n/2}^2\,n}\le 128\log(2/x),
\qquad
\frac{256\log(2/x)}{n}\le 128\,\eps_{n/2}^2\log(2/x).
\]
Thus, we have
\[
\E\big[\kl(p^\star,\widehat{p})\big]
\le
30\,\eps_{n/2}^2\int_0^1\Bigl(883+128\log(2/x)\Bigr)\,
\log\left(\alpha\Bigl(882+128\log(2/x)\Bigr)\right)\,dx
+\frac{\log(N_S)}{n/2+1}.
\]
Evaluating the integral yields
\[
\E\big[\kl(p^\star,\widehat{p})\big]
\le
33000\,\eps_{n/2}^2\log(\alpha)+231240\,\eps_{n/2}^2+\frac{\log(N_S)}{n/2+1}.
\]
Since $\alpha\ge 2$, we have $\log(\alpha)\ge \log 2$, so
$231240\,\eps_{n/2}^2 \le \frac{231240}{\log 2}\,\eps_{n/2}^2\log(\alpha)$.
This proves the claim.\qedhere
\end{proof}

When combined with the covering estimates of \Cref{thm:local-entropy-mixtures} and \Cref{thm:ratio-cover}, this establishes the optimal in-expectation rate for mixture density estimation:

\begin{corollary}\label{cor:expected-mix}
In the setup of \Cref{thm:convexaggregation} let $n$ be even.
There exists a density estimator $\widehat{p}$ such that
\[
\E\big[\kl(p^\star,\widehat{p})\big]
\le
11000000 \, \frac{M}{n}.
  % \,M}{n}.
\]
\end{corollary}

\begin{proof}
Apply Proposition~\ref{prop:expected-recipe} with $\mathcal{P}=\conv\{p_1,\ldots,p_\M\}$ and $\alpha=32$.
By Theorem~\ref{thm:local-entropy-mixtures}, $\nloc(\mathcal{P},\eps)\le 64^{\M-1}$ for all $\eps>0$, so we may take
\[
\eps_{n/2}^2=\frac{2M\log(64)}{n},
\]
which satisfies $(n/2)\eps_{n/2}^2 \ge \log \nloc(\mathcal{P},\eps_{n/2})$. Fix any $p\in\mathcal{P}$ and write $p=p_v$ for some $v\in\Delta_{\M-1}$, where $p_v=\sum_{j=1}^M v_j p_j$.
The set of mixture weights
\[
K=\Bigl\{u\in\Delta_{\M-1}:\ p_u \in \mathcal{P}\cap B_{\hels}(p,\eps_{n/2}^2)\Bigr\}
\]
is convex and compact. By Theorem~\ref{thm:ratio-cover}, it admits a $32$-ratio cover of size at most $2^{8M}$; mapping these weights back to densities yields a $32$-ratio cover $S(p)$ of $\mathcal{P}\cap B_{\hels}(p,\eps_{n/2}^2)$ with the same size. Hence Proposition~\ref{prop:expected-recipe} applies with $N_S=2^{8M}$. We obtain
\[
\E\big[\kl(p^\star,\widehat{p})\big]
\le
366609\,\eps_{n/2}^2\log(32)+\frac{8M\log 2}{n/2+1}.
\]
Using $1/(n/2+1)\le 2/n$ and $\eps_{n/2}^2=2M\log(64)/n$, the right-hand side is at most
\[
\frac{366609\cdot 2\log(64)\cdot \log(32)}{n}\,M+\frac{16\log 2}{n}\,M
\le \frac{10568316\,M}{n},
\]
which completes the proof.\qedhere
\end{proof}

\subsection{KL covering numbers}\label{subsection:apply-kl-cover}

Constructing covers in Kullback--Leibler divergence is substantially less developed than for metrics such as Hellinger distance. 
Tang's PhD thesis \cite{tang2022divergence} highlights the KL covering number of the simplex as a natural benchmark problem. This benchmark is directly tied to discrete distribution estimation: in this section, $\Delta_{d-1}$ is the convex hull of the $d$ Dirac masses on $[d]$, equivalently the class of mixtures of a finite dictionary of $d$ point masses.

We note that the results of this section are mainly illustrative. We focus on the simplex to mirror the setting studied in \cite{tang2022divergence}, but global KL covering numbers for the simplex are too crude for our statistical applications, where we crucially exploit localization. More broadly, our ratio-cover approach applies to convex hulls of arbitrary collections of densities, and our focus is always on {local} covering numbers.

Since KL is a directed divergence, we use the one-sided notion of covering that matches our risk $\kl(p^\star,\widehat p)$ and is the notion used in the work of Yang and Barron \cite{yang1999information}: for $\eps>0$,
\[
N_{\kl}(\Delta_{d-1},\eps)
\triangleq
\min\Bigl\{N:\ \exists\,q^{(1)},\dots,q^{(N)}\in\Delta_{d-1}\ \text{such that}\ 
\sup_{p\in\Delta_{d-1}}\ \min_{k\in[N]}\ \kll{p}{q^{(k)}}\le \eps\Bigr\}.
\]

\paragraph{Tang's KL covers for the simplex \cite{tang2022divergence}.}
Tang gives a general lower bound (via a reduction to total variation covering numbers), which yields
\begin{equation}\label{eq:kl-cover-lb}
N_{\kl}(\Delta_{d-1},\eps) \ge \left( \frac{1}{8 \eps}\right)^{\frac{d-1}{2}}.
\end{equation}
On the other hand, for any $p\in\Delta_{d-1}$, the uniform distribution $u=(1/d,\dots,1/d)$ satisfies $\kll{p}{u}\le \log d$, so $N_{\kl}(\Delta_{d-1},\eps)=1$ whenever $\eps\ge \log d$.
For $0<\eps<\log d$, Tang's Theorem~2 gives the upper bound
\[
N_{\kl}(\Delta_{d-1},\eps) \le \left( C \frac{\log d}{\eps} \right)^{\frac{d-1}{2}}
\]
for a universal constant $C>0$, using a Hellinger cover together with a uniform lower bound on coordinates to control likelihood ratios.
In the small-$\eps$ regime $0<\eps<\frac{1}{4(d+1)^2}$, Tang's Theorem~3 gives the sharper bound
\[
N_{\kl}(\Delta_{d-1},\eps) \le \left( \frac{C}{\eps} \right)^{\frac{d-1}{2}} \log \frac{1}{\eps},
\]
via a probabilistic construction, namely the proof uses a $\chi^2$-type argument based on a modification of Jeffreys prior. These results motivate the following question: for a fixed accuracy level (say $0<\eps\le 1$), can one match the lower bound
\eqref{eq:kl-cover-lb} up to universal constants, i.e., obtain an upper bound of the form
\[
N_{\kl}(\Delta_{d-1},\eps) \le \left(\frac{C}{\eps}\right)^{\frac{d-1}{2}}
\qquad (0<\eps\le 1)
\]
for some universal constant $C>0$?
We answer this question in the affirmative using ratio covers (see \Cref{prop:kl-cover} below).
We note that Tang also analyzes additional regimes where $\eps$ varies with $d$, including subexponential and polynomial behaviors; for an overview we reproduce their phase diagram in Table~\ref{table:tang}.

{
\begin{table}[htbp]
\scriptsize
\centering
\renewcommand{\arraystretch}{1.7}
\setlength{\tabcolsep}{10pt}
\begin{tabular}{|l|l|c|c|}
\hline
\multicolumn{4}{|c|}{\textbf{ Simplex KL Covering Number Bounds of \cite{tang2022divergence}}}\\
\hline
\textbf{Name} & \textbf{Regime} & \textbf{Lower Bound} & \textbf{Upper Bound}\\
\hline
Hellinger method
& $0<\varepsilon<\log d$
& $\left(\dfrac{c}{\varepsilon}\right)^{\frac{d-1}{2}}$
& $\left(\dfrac{C\log d}{\varepsilon}\right)^{\frac{d-1}{2}}$
\\
\hline
$\chi^2$ method
& $0<\varepsilon<\dfrac{1}{4(d+1)^2}$
& $\left(\dfrac{c}{\varepsilon}\right)^{\frac{d-1}{2}}$
& $\log \left(\dfrac{1}{\varepsilon}\right)
  \left(\dfrac{C}{\varepsilon}\right)^{\frac{d-1}{2}}$
\\
\hline
Subexponential
& $\displaystyle \varepsilon=\frac{1}{r}\log d+\log(r+1),\; 1\le r\le e\log d$
& $\exp\left(\frac{1}{e^{2}} d^{\,1-1/r}\right)$
& $\exp({d^{\,1-1/r} \log d})$
\\
\hline
Polynomial
& $\displaystyle \varepsilon=\log\frac{d}{\gamma},\quad
   \gamma\le\sqrt{\frac{d}{2}}$
& $\exp({\frac{\gamma}{e} + \log \sqrt{\frac{2\pi\gamma}{e}}})$
& $\exp(2\gamma \log d)$
\\
\hline
\end{tabular}
\caption{Covering number bounds for the simplex in \cite{tang2022divergence} for different regimes of $\varepsilon$. Constants $c$ and $C$ may vary for each bound. This table corresponds to their Figure 3.2. Our ratio cover results will improve the regimes affected by the Hellinger method and the $\chi^2$ method.\label{table:tang}}
\end{table}
}

Our next result improves the KL covering number upper bound in the regime where $\eps$ is at most an absolute constant. We also state a general reduction from KL covers of the simplex to local Hellinger covers, combined with ratio covers.

\begin{proposition}\label{prop:kl-cover}
Let $N_{\hels}(\Delta_{d-1},r)$ denote the covering number of the simplex $\Delta_{d-1}$ in squared Hellinger distance at scale $r$.
Then, for every $\eps>0$,
\[
N_{\kl}(\Delta_{d-1},\eps)\le N_{\hels}\left(\Delta_{d-1},\frac{\eps}{70}\right)\cdot 2^{8d}.
\]
In particular, for every $0<\eps\le 1$,
\[
N_{\kl}(\Delta_{d-1},\eps)\le\left(\frac{5040\cdot 2^{32}}{\eps}\right)^{\frac{d-1}{2}}.
\]
\end{proposition}

\begin{proof}
If $d = 1$, then the set is the singleton and the proof is immediate. Thus, we assume $d \ge 2$. Fix $\eps>0$ and set $R\triangleq \eps/70$.
Let $\{q^{(1)},\dots,q^{(N)}\}\subset\Delta_{d-1}$ be an $R$-cover of $\Delta_{d-1}$ in squared Hellinger distance, so that
for every $p\in\Delta_{d-1}$ there exists $k\in[N]$ with $\hels(p,q^{(k)})\le R$, and $N=N_{\hels}(\Delta_{d-1},R)$. For each $k\in[N]$, define the (squared-Hellinger) ball intersection
\[
K_k \triangleq \Delta_{d-1}\cap B_{\hels}(q^{(k)},R)
\qquad\text{where}\qquad
B_{\hels}(q,R)\triangleq\{p\in\Delta_{d-1}:\ \hels(p,q)\le R\}.
\]
Applying \Cref{thm:ratio-cover} to a convex and compact set $K_k$, we obtain a $32$-ratio cover $A_k\subset K_k$ with $|A_k|\le 2^{8d}$.
Define $\mathcal A\triangleq \bigcup_{k=1}^N A_k$, so
\[
|\mathcal A|\le N_{\hels}(\Delta_{d-1},R)\cdot 2^{8d}.
\]

Now fix $p\in\Delta_{d-1}$ and pick $k\in[N]$ such that $\hels(p,q^{(k)})\le R$.
Since $p\in K_k$ and $A_k$ is a $32$-ratio cover of $K_k$, there exists $q^\star\in A_k$ such that $p_i\le 32\,q^\star_i$ for all $i\in[d]$.
In particular, whenever $p_i>0$ we have $q^\star_i>0$ and $\log(p_i/q^\star_i)\le \log 32$, hence
$
\max_{i\in[d]}\log\frac{p_i}{q^\star_i}\le \log 32.
$
Moreover, since $q^\star\in A_k\subset K_k\subset B_{\hels}(q^{(k)},R)$, we also have $\hels(q^{(k)},q^\star)\le R$.
Using the triangle inequality for $\hel$,
\[
\hel(p,q^\star)\le \hel\bigl(p,q^{(k)}\bigr)+\hel\bigl(q^{(k)},q^\star\bigr)
\le\sqrt{R}+\sqrt{R}=2\sqrt{R},
\]
and thus $\hels(p,q^\star)=\hel(p,q^\star)^2\le 4R$.
By \Cref{lem:kltohellinger},
\[
\kll{p}{q^\star}
\le\frac{2}{(\sqrt e-1)^2}\hels(p,q^\star)\max\left\{1,\max_{i\in[d]}\log\frac{p_i}{q^\star_i}\right\}
\le\frac{8}{(\sqrt e-1)^2}\max\{1,\log 32\}\,R.
\]
The numerical constant satisfies $\frac{8}{(\sqrt e-1)^2}\max\{1,\log 32\}<70$, so with $R=\eps/70$ we have $\kll{p}{q^\star}\le \eps$.
Since $p\in\Delta_{d-1}$ was arbitrary, $\mathcal A$ is a KL $\eps$-cover of $\Delta_{d-1}$, and therefore
\[
N_{\kl}(\Delta_{d-1},\eps)\le N_{\hels}\left(\Delta_{d-1},\frac{\eps}{70}\right)\cdot 2^{8d}.
\]

It remains to derive the stated $(C/\eps)^{(d-1)/2}$ bound when $0<\eps\le 1$.
Consider the bijection $p\mapsto \sqrt p\triangleq(\sqrt{p_1},\dots,\sqrt{p_d})$ between $\Delta_{d-1}$ and the positive orthant of the unit sphere
$
{S}^{d-1}_+\triangleq \{x\in\R_+^d:\ \|x\|_2=1\}.
$
For $p,q\in\Delta_{d-1}$, we have
$
\hels(p,q) = \frac12\bigl\|\sqrt p-\sqrt q\bigr\|_2^2,
$ 
which implies
\[
N_{\hels}(\Delta_{d-1},r)= N_{\ell_2}\left({S}^{d-1}_+,\sqrt{2r}\right)
\le N_{\ell_2}\left({S}^{d-1},\sqrt{2r}\right),
\]
where $N_{\ell_2}({S}^{d-1},t)$ denotes the $\ell_2$ covering number of the unit sphere at radius $t$.
For $0<t\le 1$, a standard shell-volumetric argument yields
$
N_{\ell_2}({S}^{d-1},t) \le \left(\frac{12}{t}\right)^{d-1}.
$
Fix $r=\eps/70$. For $0<\eps\le 1$, we have $\sqrt{2r}=\sqrt{2\eps/70}\le 1$, hence
\[
N_{\hels}\left(\Delta_{d-1},\frac{\eps}{70}\right)
\le\left(\frac{5040}{\eps}\right)^{\frac{d-1}{2}}.
\]
Combining this with the first part of the proof, we prove the claim.
\end{proof}

\section{Aggregation in the misspecified case}
\label{sec:misspecified}
In this section we discuss aggregation under misspecification, where we do not assume that $p^\star \in \mathcal{P}$. This problem is also of interest in the aggregation literature. The results in this section have the following features:
\begin{itemize}
\item we do not assume that $p^\star \in \mathcal{P}$;
\item the algorithms are simpler, and their analysis relies on less sophisticated arguments;  
\item however, their statistical guarantees feature additional polylogarithmic factors in $n$ and $\M$.
  % at the price of additional polylogarithmic factors in $n$ and $\M$ in the bounds.
\end{itemize}
In particular,
obtaining optimal rates up to universal constant factors when $p^\star \notin \mathcal{P}$ remains an open question.

\begin{proposition}
\label{prop:misspecmodelaggregation}
Let $X_{1}, \ldots, X_n$ be an i.i.d.\ sample from some unknown density $p^\star$ defined on the space $\X$ with respect to the measure $\mu$. Let $\mathcal{P} = \{p_1, \ldots, p_\M\}$ be a set of densities with respect to $\mu$ defined on $\X$. There is a density estimator $\widehat{p}$ such that for any $\delta \in (0, 1/2]$, with probability at least $1 - \delta$,
\[
\kl\left(p^\star, \widehat{p}\right)
\le \min_{j \in [M]}\kl\left(p^\star, p_j\right)
+ \frac{19\log(2nM)\log(1/\delta)}{n}.
\]
\end{proposition}
Compared to \Cref{thm:modelaggregation} the bound contains an additional dependence on $n$ under the logarithm. The estimator and the analysis are presented in \Cref{sec:remainingproofs}. The idea behind the estimator is to run the predictor of \Cref{lem:onlinetobatch} on a deterministically smoothed class $\Qs$ obtained from the given set $\mathcal{P}$.

Our next result is an analog of \Cref{thm:convexaggregation} in the misspecified case.
\begin{proposition}
\label{prop:misspecconvexaggregation}
    Let $X_{1}, \ldots, X_n$ with $n \ge 2$ be an i.i.d.\ sample from some unknown density $p^\star$. Let $\mathcal{P} = \{p_1, \ldots, p_\M\}$ be a set of densities with respect to $\mu$ defined on $\X$. There is a density estimator $\widehat{p}$ such that, with probability at least $1 - \delta$,
    \[
    \kl\left(p^\star, \widehat{p}\right) \le \inf\limits_{p \in \operatorname{conv}(\mathcal{P})}\kl\left(p^\star, p\right) + \frac{43M\log^2(3nM) + 22\log(3nM)\log(1/\delta)}{n}.
    \]
\end{proposition}

The main difference with \Cref{thm:convexaggregation} consists in additional squared logarithmic terms in the risk bound. 
The estimator in \Cref{prop:misspecconvexaggregation} corresponds to the MLE over a smoothed density class:
\begin{framed}
\begin{itemize}
\item Construct a class of densities 
\[
\Qs
=
\bigg\{
\left(1 - \frac{1}{n}\right)p
+ \frac{1}{nM}\sum_{q \in \mathcal{P}} q
:\ p \in \operatorname{conv}(\mathcal{P})
\bigg\}.
\]
\item Let $\widehat{p}$ be the MLE in $\Qs$. That is,
\[
\widehat{p} = \arg\min_{p \in \Qs}\wh L(p).
\]
\end{itemize}
\end{framed}

The proof of \Cref{prop:misspecconvexaggregation} follows relatively standard lines using techniques from empirical process theory and we defer it to \Cref{sec:remainingproofs}. 

\section{Acknowledgments}
S.C.~is supported by the NDSEG Fellowship Program, Tselil Schramm’s NSF CAREER Grant no. 2143246, and Gregory Valiant’s Simons Foundation Investigator Award and NSF award AF-2341890.
J.M.~is supported by a grant of the French National Research Agency (ANR), ``Investissements d’Avenir'' (LabEx Ecodec/ANR-11-LABX-0047).

\bibliographystyle{abbrv}
{
  \small
\bibliography{mybib}

\begin{thebibliography}{1}

\bibitem{placeholder}
A.~Author.
\newblock A placeholder reference.
\newblock \emph{Journal of Examples}, 1(1):1--2, 2025.

\end{thebibliography}
}
\appendix

\section{High-probability online-to-batch conversion}
\label{sec:highprobabilityonlinetobatch}
We present a self-contained version of the argument of Theorem 1 in \cite{vanderhoeven2023high}, adapted to the logarithmic loss and a finite class of densities. 

Consider $\M$ distributions $p_1, \ldots, p_\M$ with corresponding densities $p_1, \ldots, p_\M$. Given the sample $X_1, \ldots, X_n$ define recursively the following sequence of densities:
\[
\widehat{p}_{1} = \frac{1}{\M}\sum\limits_{i = 1}^M p_i,
\]
and for $j \in \{1, 2, \ldots, n-1\}$ set
\begin{equation}
\label{eq:individualdensities}
\widehat{p}_{j+1} = \sum\limits_{i = 1}^Mp_i \frac{\prod_{k = 1}^{j}\left(\frac{1}{2}p_i(X_k) + \frac{1}{2}\widehat{p}_{k}(X_k)\right)}{\sum\nolimits_{h = 1}^M\prod_{k = 1}^{j}\left(\frac{1}{2}p_h(X_k) + \frac{1}{2}\widehat{p}_{k}(X_k)\right)},
\end{equation}
The final predictor will be given by 
\begin{equation}
\label{eq:progressivemixturerule}
\widehat{p} = \frac{1}{n}\sum\limits_{i = 1}^n\widehat{p}_{i},
\end{equation}
which is a version of the classical progressive Bayesian mixture estimator \cite{barron1987bayes, yang1999information} with the key difference that in \eqref{eq:progressivemixturerule} for each individual weight in \eqref{eq:individualdensities} we consider the product of shifted elements $\left(\frac{1}{2}p_i(X_k) + \frac{1}{2}\widehat{p}_{k}(X_k)\right)$ instead of the canonical product of $p_i(X_k)$ in the original formulation.

We begin by providing the following lemma which provides the regret bound for the sequence $\widehat{p}_{1}, \ldots, \widehat{p}_{n}$.

\begin{lemma}
\label{eq:regretboundforshiftedloss}
Let $p_1,\ldots,p_\M$ be densities and let $X_1,\ldots,X_n$ be any sequence of points.
For the sequence $(\widehat p_t)_{t=1}^n$ defined by \eqref{eq:individualdensities} and
\eqref{eq:progressivemixturerule}, it holds that
\begin{equation}\label{eq:shifted-regret}
\sum_{t=1}^{n}-\log\big(\widehat{p}_t(X_t)\big)
\le
\min_{j\in[M]}\sum_{t=1}^{n}-\log\left(\tfrac{1}{2}\widehat{p}_t(X_t)+\tfrac{1}{2}p_j(X_t)\right)+\log M .
\end{equation}
\end{lemma}

\begin{proof}
The proof repeats the standard lines for the logarithmic loss \cite{vovk1990aggregating}. See the proof of the more general Proposition 1 in \cite{vanderhoeven2023high}.
\end{proof}

We also need the following elementary result.
\begin{lemma}
\label{lem:logloss-Tstyle}
Let $m>0$ and $x, y \in \R^+$ such that 
$\left|\log\left(x/y\right)\right|\le m$.
Then, 
\begin{equation}\label{eq:mid-log}
-\log\left(\frac{x+y}{2}\right)
\le
\frac{-\log x-\log y}{2}
-
\frac{\left(\log\left(x/y\right)\right)^2}{8(m+1)}.
\end{equation}
\end{lemma}

\begin{proof}
Let $h(w)=-\log w$. Then $h'(w)=-w^{-1}$ and $h''(w)=w^{-2}=(h'(w))^2$.
Fix $(x,y)\in \R^2$ and set $s=\frac{x+y}{2}$. Define
\[
g(w)=h(w)-h(s)-\frac{\left(h(w)-h(s)\right)^2}{\gamma}.
\]
A direct computation gives
\[
g''(w)=h''(w)\left(1-\frac{2}{\gamma}\left(h(w)-h(s)+1\right)\right).
\]
Since $\left|\log\left(x/y\right)\right|\le m$, we have $\left|h(x)-h(y)\right|\le m$. For any $w$ between $x$ and $y$, the monotonicity of $h$ implies $h(w)$ and $h(s)$ lie between $h(x)$ and $h(y)$, hence $h(w)-h(s)\le \left|h(x)-h(y)\right|\le m$. With $\gamma=2(m+1)$ we obtain $g''(w)\ge 0$ on the segment between $x$ and $y$. By convexity, $g(s)\le \frac{1}{2}g(x)+\frac{1}{2}g(y)$. Since $g(s)=0$, this yields
\[
h(s)\le \frac{h(x)+h(y)}{2}-\frac{\left(h(x)-h(s)\right)^2+\left(h(y)-h(s)\right)^2}{2\gamma}.
\]
Using $\left(a-c\right)^2+\left(b-c\right)^2\ge \frac{1}{2}\left(a-b\right)^2$ with $a=h(x)$, $b=h(y)$, $c=h(s)$, we get
\[
h\left(\frac{x+y}{2}\right)\le \frac{h(x)+h(y)}{2}-\frac{\left(h(x)-h(y)\right)^2}{4\gamma}.
\]
Substituting $h(w)=-\log w$ and $\gamma=2(m+1)$ proves the claim.
\end{proof}
Another useful lemma is a simple version of Freedman's inequality. For the version with these explicit constants we refer to \cite[Theorem 1]{beygelzimer2011contextual}. Throughout this section we denote $\E_{i - 1}[\cdot] = \E[\cdot |X_1, \ldots, X_{i - 1}]$.

\begin{lemma}
\label{lem:freedman}
Let $X_1, \ldots, X_n$ be a martingale difference sequence adapted to a filtration $(\mathcal{F}_i)_{i \le n}$. Suppose that $|X_i| \le m$ almost surely. Then for any $\delta \in (0, 1)$ and $\lambda \in [0, 1/m]$, with probability at least $1 - \delta$,
\[
\sum\limits_{i = 1}^n X_i \le \lambda(e - 2)\sum\limits_{i = 1}^n\E_{i - 1}[X_i^2] + \frac{\log(1/\delta)}{\lambda}.
\]
\end{lemma}

Finally, we restate \Cref{lem:onlinetobatch} and specify the predictor $\widehat{p}$.
\begin{theorem}
  Assume $X_{1},\ldots,X_{n}$ are
\iid random variables with arbitrary and unknown density $p^\star$ on $(\X, \mu)$.
Let $p_{1},\ldots,p_{\M}$ be densities such that for some $m>0$,
\[
\sup_{x \in \X}\max_{i,j\in[M]}\left|\log\left(\frac{p_i(x)}{p_j(x)}\right)\right|\le m.
\]
Let $\left(\widehat p_i\right)_{i=1}^{n}$ and $\widehat p=\frac{1}{n}\sum_{i=1}^{n}\widehat p_i$ be defined by \eqref{eq:individualdensities} and \eqref{eq:progressivemixturerule}. Then for every $\delta\in(0,1)$, with probability at least $1-\delta$,
\[
\kll{p^\star}{\widehat{p}} \le \min_{j \in [M]}\kll{p^\star}{p_j} + \frac{2\log\left(M\right)+\frac{25}{4}\left(e-2\right)\left(m+1\right)\log\left(\frac{1}{\delta}\right)}{n} .
\]
\end{theorem}

\begin{proof}
Fix $j\in[M]$. Write $\ell_i(q)=-\log\left(q\left(X_i\right)\right)$ and $r_i=\ell_i\left(\widehat p_i\right)-\ell_i\left(p_j\right)=\log\left(\frac{p_j\left(X_i\right)}{\widehat p_i\left(X_i\right)}\right)$. Since $\widehat p_i\left(x\right)$ is a convex combination of $\{p_k\left(x\right)\}_{k=1}^M$ and $\left|\log\left(p_a\left(x\right)/p_b\left(x\right)\right)\right|\le m$ for all $a,b \in [M]$, one has $\exp(-m)\le \widehat p_i\left(x\right)/p_j\left(x\right)\le \exp(m)$ and hence $\left|r_i\right|\le m$ almost surely. Set
\[
v_i=\frac{r_i^2}{4\left(m+1\right)},\quad \textrm{and} \quad Z_i=r_i+v_i.
\]
By construction, $\widehat p_i$ is $\mathcal F_{i-1}$-measurable and $X_i\stackrel{d}{=}X$ is independent of $\mathcal F_{i-1}$. By Jensen's inequality,
\begin{equation}\label{eq:risk-mart}
-\E\log\left(\widehat p(X)\right)+\E\log\left(p_j(X)\right)\le \frac{1}{n}\sum_{i=1}^{n}\E_{i-1}\left[ r_i\right].
\end{equation}
Apply Lemma~\ref{lem:logloss-Tstyle} with $x=\widehat p_i\left(X_i\right)$ and $y=p_j\left(X_i\right)$ we get
\[
-\log\left(\frac{\widehat p_i\left(X_i\right)+p_j\left(X_i\right)}{2}\right)\le \frac{\ell_i\left(\widehat p_i\right)+\ell_i\left(p_j\right)}{2}-\frac{r_i^2}{8\left(m+1\right)}.
\]
Equivalently,
\begin{equation}\label{eq:one-step}
r_i\le 2\left(\tilde\ell_i\left(\widehat p_i\right)-\tilde\ell_i\left(p_j\right)\right)-\frac{r_i^2}{4\left(m+1\right)}=2\left(\tilde\ell_i\left(\widehat p_i\right)-\tilde\ell_i\left(p_j\right)\right)-v_i,
\end{equation}
where $\tilde\ell_i\left(q\right)=-\log\left(\frac{\widehat p_i\left(X_i\right)+q\left(X_i\right)}{2}\right)$. Summing \eqref{eq:one-step} over $i$ and using Lemma~\ref{eq:regretboundforshiftedloss} (which, for each $j$, gives $\sum_{i=1}^{n}\left(\tilde\ell_i\left(\widehat p_i\right)-\tilde\ell_i\left(p_j\right)\right)\le \log\left(M\right)$) yields
\begin{equation}\label{eq:sum-r}
\sum_{i=1}^{n}r_i\le 2\log\left(M\right)-\sum_{i=1}^{n}v_i.
\end{equation}
Consider the martingale differences $Y_i=\E_{i-1}\left[Z_i\right]-Z_i$ with respect to $\left(\mathcal F_i\right)_{i=1}^{n}$. Since $\left|r_i\right|\le m$, one has $\left|Z_i\right|\le \left|r_i\right|+\frac{r_i^2}{4\left(m+1\right)}\le m+\frac{m^2}{4\left(m+1\right)}$. Denote $\widetilde{m} = m+\frac{m^2}{4\left(m+1\right)}$. We have $\left|Y_i\right|\le 2\widetilde{m}$ almost surely. Moreover, since $|r_i|\le m$ we have
$
Z_i=r_i+\frac{1}{4(m+1)}r_i^2=r_i\Bigl(1+\frac{1}{4(m+1)}r_i\Bigr),
$
which implies
$
Z_i^2=r_i^2\Bigl(1+\frac{1}{4(m+1)}r_i\Bigr)^2\le r_i^2\Bigl(1+\frac{1}{4(m+1)}m\Bigr)^2=c_m^2 r_i^2,
$
where $c_m=\frac{5m+4}{4(m+1)}$. Applying Lemma \ref{lem:freedman} with 
\[
\lambda=\frac{1}{4\left(m+1\right)\left(e-2\right)c_m^2}\le \frac{1}{2\widetilde{m}},
\]
we obtain, with probability at least $1-\delta$,
\[
\sum_{i=1}^{n}\left(\E_{i-1}\left[Z_i\right]-Z_i\right)\le \lambda\left(e-2\right)\sum_{i=1}^{n}\E_{i-1}\left[Y_i^2\right]+\frac{\log\left(1/\delta\right)}{\lambda}
\le \sum_{i=1}^{n}\E_{i-1}\left[v_i\right]+\frac{\log\left(1/\delta\right)}{\lambda},
\]
where in the last inequality we used $\E_{i-1}[Y_i^2] \le \E_{i - 1}[Z_i^2] \le 4(m+1)c_m^2\E_{i - 1}[v_i]$.
Expanding $Z_i=r_i+v_i$ and canceling $\sum_{i}\E_{i-1}\left[v_i\right]$ gives
\begin{equation}\label{eq:freedman}
\sum_{i=1}^{n}\E_{i-1}[r_i]\le \sum_{i=1}^{n}r_i+\sum_{i=1}^{n}v_i+\frac{\log\left(1/\delta\right)}{\lambda}.
\end{equation}
Combining \eqref{eq:freedman} with \eqref{eq:sum-r} yields
\[
\sum_{i=1}^{n}\E_{i-1}[r_i]\le 2\log\left(M\right)+\frac{\log\left(1/\delta\right)}{\lambda}.
\]
Inserting this into \eqref{eq:risk-mart} and finally choosing $j\in[M]$ that minimizes $\left(-\E\log\left(p_j(X)\right)\right)$, we have
\[
-\E\log\left(\widehat p(X)\right)-\min_{j\in[M]}\left(-\E\log\left(p_j(X)\right)\right)\le \frac{2\log\left(M\right)}{n}+\frac{\log\left(1/\delta\right)}{n\,\lambda}.
\]
Finally, $1/\lambda=4\left(m+1\right)\left(e-2\right)c_m^2=\frac{\left(e-2\right)\left(5m+4\right)^2}{4\left(m+1\right)}\le \frac{25}{4}\left(e-2\right)\left(m+1\right)$, which gives \eqref{eq:main-theorem-bound}. The claim follows.
\end{proof}

\section{Suboptimality of MLE and
  Bayesian model averaging}
\label{sec:suboptimal}

In this section, we show that simple baseline methods for model aggregation, such as MLE and Bayes model averaging, fail to achieve optimal guarantees in the distribution-free setting.

The following simple fact, mentioned in the introduction, shows that in the distribution-free setting, no selection procedure (such as MLE) achieves any nontrivial guarantee.

\begin{fact}
  \label{fac:model-selection-fails}
  For every $n \geq 1$, there exists a dictionary $\probas = \set{p_1, p_2}$ such that the following holds.
  For every estimator $\wh p = \wh p_n$ based on an \iid sample $X_1, \dots, X_n$ of size $n$ and taking values inside $\probas$, there exists $p^\star \in \probas$ such that if $X_1, \dots, X_n \sim p^\star$, then
  $\P (\kll{p^\star}{\wh p_n} = + \infty) \geq 0.99$.
\end{fact}

\begin{proof}
  Let $\X = \set{0, 1, 2}$ and $\eta > 0$ be such that $(1-\eta)^n = 0.99$.
  Define the densities $p_1, p_2$ on $\X$ (with respect to the counting measure) by $p_1 (0) = 1-\eta, p_1 (1) = \eta, p_1 (2) = 0$  and $p_2 (0) = 1-\eta, p_2 (1) = 0, p_2 (2) = \eta$.
  Let $E$ denote the event $\set{X_1 = \dots = X_n = 0}$ and $q$ the value taken by the estimator $\wh p_n$ under $E$, such that (for both $p^\star = p_1$ and $p^\star = p_2$) $\P (\wh p_n = q) \geq \P (E) = (1-\eta)^n = 0.99$.
  Now if $q = p_1$ we take $p^\star = p_2$, and if $q = p_2$ we take $p^\star = p_1$; in both cases, $\kll{p^\star}{q} = + \infty$ and thus $\P (\kll{p^\star}{\wh p_n} = + \infty) \geq 0.99$.
\end{proof}

Next, in the convex aggregation setting, Theorem~2 of \cite{mourtada2025estimation} shows that any estimator that does not depend on $\delta$ must incur a suboptimal high-probability error bound. In particular, there exist instances (already with $\M=2$) for which any expectation-optimal estimator, including the progressive mixture estimator of Yang and Barron (formally defined in \Cref{thm:yang-barron}), incurs a lower bound of order
\[
\frac{\log(1/\delta)\log\log(1/\delta)}{n},
\]
which is worse than what is achieved by our $\delta$-dependent estimator in \Cref{thm:convexaggregation}.

In the remainder of this section, we show that an even stronger lower bound holds for
Bayesian model averaging, which is arguably the most natural model aggregation procedure.
For completeness, we also provide a matching upper bound for the case of two densities, showing that our lower bound is tight up to absolute constants. Our lower bound focuses on two Bernoulli distributions, as in the lower-bound construction of \cite{mourtada2025estimation}.

Our next result shows that Bayesian model averaging, while satisfying nontrivial risk bounds, fails to achieve the optimal dependence on $\delta$ even when $\M=2$. In particular, they necessarily incur a $\log^2(1/\delta)$ factor instead of the optimal $\log(1/\delta)$ factor.

\begin{proposition}
\label{prop:bayes-two}
Let $p^\star,q$ be two densities on $\X$ with respect to $\mu$, and let $X_1,\ldots,X_n$ be i.i.d.\ from $p^\star$. Consider Bayesian model averaging with a uniform prior, which outputs
\[
\widehat p=(1-\widehat\alpha)p^\star+\widehat\alpha\,q,\qquad
\widehat\alpha=\frac{\prod_{i=1}^n q(X_i)}{\prod_{i=1}^n p^\star(X_i)+\prod_{i=1}^n q(X_i)}.
\]
Then for any $\delta\in(0,1/2)$, with probability at least $1-\delta$,
\[
\kl(p^\star,\widehat p)\le \min\left\{\frac{10\log^2(2/\delta)}{n}, 5\log\left(\frac{2}{\delta}\right)\right\}.
\]
Moreover, taking $\X=\{0,1\}$, letting $q$ be the point mass at $1$, and letting $p^\star$ be the Bernoulli law with $(1-p^\star(0))^n=\delta$, we have, for any $\delta\in(0,e^{-2}]$ and with probability at least $\delta$,
\[
\kl(p^\star,\widehat p)\ge \min\left\{\frac{\log^2(1/\delta)}{4n},\, \frac{1}{4}\log\left(\frac{1}{\delta}\right)\right\}.
\]
\end{proposition}

\begin{proof}
By Markov's inequality and $\E\big[\prod_{i=1}^n q(X_i)/p^\star(X_i)\big]\le 1$,
\[
\P\left(\frac{\prod_{i=1}^n q(X_i)}{\prod_{i=1}^n p^\star(X_i)}\ge \frac{1}{\delta}\right)\le \delta.
\]
Hence, on an event of probability at least $1-\delta$ we have $\frac{\prod q(X_i)}{\prod p^\star(X_i)}<1/\delta$, which implies
\[
1-\widehat\alpha=\frac{1}{1+\frac{\prod q(X_i)}{\prod p^\star(X_i)}}\ge \frac{\delta}{1+\delta}\ge \frac{\delta}{2}.
\]
Therefore, for all $x$, 
\[
\widehat p(x)=(1-\widehat\alpha)p^\star(x)+\widehat\alpha q(x)\ge (1-\widehat\alpha)p^\star(x)\ge \frac{\delta}{2}p^\star(x),
\]
and so
\begin{equation}
\label{eq:env-bayes}
\sup_{x}\log\left(\frac{p^\star(x)}{\widehat p(x)}\right)\le \log\left(\frac{2}{\delta}\right).
\end{equation}
By Lemma~\ref{lem:kltohellinger},
\begin{equation}
\label{eq:hellingerbound}
\kl(p^\star,\widehat p)\le 5\hels(p^\star,\widehat p)\log\left(\frac{2}{\delta}\right).
\end{equation}
We consider the cases. First, if $2\log(2/\delta)/n > 1$, then the minimum in the statement is achieved at the second term and since $\hels(p^\star,\widehat p) \le 1$, we have
\[
\kl(p^\star,\widehat p)\le 5\log\left(\frac{2}{\delta}\right),
\]
which proves the claim. Otherwise, we have $2\log(2/\delta)/n \le 1$, and we again consider two cases. In the first case, assume $\hels(p^\star,q)\le \frac{2\log(2/\delta)}{n}$.
Using \eqref{eq:env-bayes} and the convexity of the squared Hellinger distance in its second argument, $\hels(p^\star,\widehat p)\le \hels(p^\star,q)$, which implies
\[
\kl(p^\star,\widehat p)\le \frac{10\log^2(2/\delta)}{n}.
\]
Otherwise, we have $\hels(p^\star,q) > \frac{2\log(2/\delta)}{n}$. In this case we want to upper bound $\widehat{\alpha}$. We have as in the proof of \Cref{lem:almostmle} for any $t > 0$,
\[
\P\left(
\frac{\prod_{i=1}^n q(X_i)}{\prod_{i=1}^n p^\star(X_i)} \ge t
\right)
\le
\frac{1}{\sqrt{t}}\bigl(1-\hels(p^\star,q)\bigr)^n
\le
\frac{1}{\sqrt{t}} \exp\bigl(-n\hels(p^\star,q)\bigr).
\]
We choose $t = \frac{2\log(2/\delta)}{n\hels(p^\star,q)}$. Note that $t<1<1/\delta$. By the above line we have
\[
\P\left(
\frac{\prod_{i=1}^n q(X_i)}{\prod_{i=1}^n p^\star(X_i)}
\ge
\frac{2\log(2/\delta)}{n\hels(p^\star,q)}
\right)
\le
\sqrt{\frac{n\hels(p^\star,q)}{2\log(2/\delta)}}\exp(-n\hels(p^\star,q))
\le
\exp(-2\log(2/\delta))
\le
\delta,
\]
where we used that $\sqrt{x}\exp(-x)$ is decreasing for $x \ge 1/2$ and $n\hels(p^\star,q) > 2\log(2/\delta) > 1/2$. Therefore, on the complementary event,
$
\widehat{\alpha}
\le
\frac{\prod_{i=1}^n q(X_i)}{\prod_{i=1}^n p^\star(X_i)}
<
\frac{2\log(2/\delta)}{n\hels(p^\star,q)},
$
and \eqref{eq:env-bayes} also holds since $t<1/\delta$. 
Hence, by \eqref{eq:hellingerbound} and the convexity of the squared Hellinger distance in its second argument, we have on this event
\[
\kl(p^\star,\widehat p)
\le
5\hels(p^\star,\widehat p)\log\left(\frac{2}{\delta}\right)
\le
5\widehat{\alpha}\hels(p^\star,q)\log\left(\frac{2}{\delta}\right)
\le
\frac{10\log^2(2/\delta)}{n}.
\]
The upper bound follows.

For the lower bound, let $\X=\{0,1\}$. Let $q$ be the point mass at $1$ (so $q(1)=1,q(0)=0$). Fix $p\in(0,1)$ such that $(1-p)^n=\delta$, and let $p^\star$ be the Bernoulli law with $p^\star(0)=p$ and $p^\star(1)=1-p$. On the event
\[
A=\{X_1=\cdots=X_n=1\},
\]
which satisfies $\P(A)=(1-p)^n=\delta$, we have
\[
\widehat\alpha=\frac{\prod_{i=1}^n q(X_i)}{\prod_{i=1}^n p^\star(X_i)+\prod_{i=1}^n q(X_i)}=\frac{1}{1+\delta},
\qquad
\widehat p(0)=(1-\widehat\alpha)p=\frac{p\delta}{1+\delta}.
\]
By the Bernoulli KL formula,
\[
\kl(p^\star,\widehat p)
=(1-p)\log\left(\frac{(1+\delta)(1-p)}{1+\delta-p\delta}\right)
+p\log\left(\frac{1+\delta}{\delta}\right).
\]
Since $1+\delta-p\delta\le 1+\delta$, the ratio $\frac{(1+\delta)(1-p)}{1+\delta-p\delta}$ lies in $[1-p,1]$. Hence
\[
(1-p)\log\left(\frac{(1+\delta)(1-p)}{1+\delta-p\delta}\right)\ge (1-p)\log(1-p)\ge -p,
\]
where the last inequality follows from $-\log(1-x)\le \frac{x}{1-x}$. Therefore, on $A$,
\begin{equation}
\label{eq:lb-core}
\kl(p^\star,\widehat p)\ge p\log\left(\frac{1+\delta}{\delta}\right)-p \ge p\log\left(\frac{1}{\delta}\right)-p.
\end{equation}
We now show that $p\ge \frac{\log(1/\delta)}{n+\log(1/\delta)}$.
Using $1-x\ge \exp\left(-\frac{x}{1-x}\right)$ for $x\in(0,1)$, we have
\[
\delta=(1-p)^n\ge \exp\left(-\frac{pn}{1-p}\right),
\]
which implies $p\ge \frac{\log(1/\delta)}{n+\log(1/\delta)}$.
Thus, if $\log(1/\delta)\le n$ then $p\ge \frac{1}{2}\frac{\log(1/\delta)}{n}$.
Since $\delta\le e^{-2}$ implies $\log(1/\delta)\ge 2$, from \eqref{eq:lb-core} we get
\[
\kl(p^\star,\widehat p)\ge p\bigl(\log(1/\delta)-1\bigr)\ge \frac{1}{2}\,p\log\left(\frac{1}{\delta}\right).
\]
Therefore, for $\delta\in(0,e^{-2}]$, it holds that
\[
\kl(p^\star,\widehat p)\ \ge\ \frac{1}{2}\,p\log\left(\frac{1}{\delta}\right)
\ \ge\ \begin{cases}
\displaystyle \frac{1}{4}\frac{\log^2(1/\delta)}{n}, & \text{if }\ \log(1/\delta)\le n,\\[1.2ex]
\displaystyle \frac{1}{4}\log\left(\frac{1}{\delta}\right), & \text{if }\ \log(1/\delta)\ge n,
\end{cases}
\]
where, in the second case, we also used $p\ge \frac{\log(1/\delta)}{n+\log(1/\delta)}\ge \frac{1}{2}$. Since $\P(A)=\delta$, these lower bounds hold with probability at least $\delta$. This completes the proof.
\end{proof}

\section{Birg\'e--Le Cam tournament}
\label{sec:blc-local-pointwise}
In this section we reproduce the details of the Birg\'e--Le Cam tournament procedure. The derivations are standard and we mainly follow the exposition in \cite[Section 32.2]{polyanskiy2025information} with the minor difference that we need to provide sharper tail bounds for our applications. For earlier results establishing this procedure we refer to \cite{lecam1973convergence, birge1983approximation, birge1985non}. For a set of densities $\mathcal P$, recall the definition \eqref{eq:localentropy} of the local Hellinger entropy.
Let $\eps_n>0$ satisfy the fixed-point inequality
\begin{equation}
\label{eq:fixedpoint}
n\eps_n^2\ge \log \nloc(\mathcal{P},\eps_n) \lor 1.
\end{equation}

\begin{lemma}[Le Cam-Birg\'e tournament]
\label{lem:lecam-birge-local}
Assume we observe an i.i.d. sample $X_1, \ldots, X_n$ distributed according to the density $p^\star$ with respect to the measure $\mu$. Let $\mathcal P$ be a known class of densities such that $p^\star \in \mathcal{P}$. 
There exists an estimator $\widehat{p}=\widehat{p}(X_1,\ldots,X_n)\in\mathcal{P}$ such that for every $\delta\in(0,1)$, with probability at least $1-\delta$,
\[
\hel(p^\star,\widehat{p})\le 21\eps_n+8\sqrt{\frac{\log(2/\delta)}{n}},
\]
where $\eps_n$ is given by \eqref{eq:fixedpoint}.
Moreover, 
\[
\E\hel(p^\star,\widehat{p}) \le 21\eps_n + \frac{11}{\sqrt{n}}.
\]
\end{lemma}

We now present a construction of the estimator as well as the construction of the convex set containing $p^\star$, which will be used in our analysis.
\begin{framed}
\begin{enumerate}
\item Take a maximal $\eps$-packing $\{p_1,\ldots,p_N\}\subset\mathcal{P}$ in Hellinger distance $\hel$.
\item For each pair $i\ne j$ with $\hel(p_i,p_j)\ge R$ where $R=4\eps$, run the composite test $\psi_{ij}$ of Lemma~\ref{lemma:pairwise-composite} between the hypotheses $B_{\hel}(p_i,\eps)$ and $B_{\hel}(p_j,\eps)$ and set $\psi_{ji}=1-\psi_{ij}$.
\item Define tournament scores
\[
T_i=
\begin{cases}
\max\limits_{j:\ \hel(p_i,p_j)\ge R,\ \psi_{ij}=1}\ \hel(p_i,p_j), & \text{if such } j \text{ exists},\\
0, & \text{otherwise.}
\end{cases}
\]
\item Let $i^\star\in\arg\min_i T_i$ (ties arbitrary) and output $\widehat{p}=p_{i^\star}$. Namely, we choose the density for which no $R$-distant center \say{wins the match} defined by the composite test $\psi_{ij}$.
\item For any $\delta\in(0,1)$ define the set
\[
\wh \Qs_1=\left\{q\in\mathcal{P}: \hel(q,\widehat{p})\le 21\eps_n+8\sqrt{\frac{\log(2/\delta)}{n}}\right\}.
\]
Note that the map $q\mapsto \hels(q,\widehat{p})$ is convex, so $\wh \Qs_1$ is convex (whenever $\mathcal{P}$ is also convex), and by \cref{lem:lecam-birge-local} it contains $p^\star$ with probability at least $1-\delta$.
\end{enumerate}
\end{framed}
We first recall the following result on hypotheses testing between Hellinger balls, due to Le Cam~\cite{lecam1973convergence} and Birgé~\cite{birge1983approximation}.
\begin{lemma}[Pairwise composite testing; equation (32.33) of \cite{polyanskiy2025information}]
\label{lemma:pairwise-composite}
Let $\eps>0$ and $R=4\eps$. For any $i\ne j$ with $\hel(p_i,p_j)\ge R$, there exists a test $\psi_{ij}(X_1,\ldots,X_n)\in\{0,1\}$ such that
the outcome $0$ means $p\in B_{\hel}(p_i,\eps)$, the outcome $1$ means $p\in B_{\hel}(p_j,\eps)$, and $\psi_{ji}=1-\psi_{ij}$.
Moreover,
\[
\sup_{p\in B_{\hel}(p_i,\eps)}\P_p[\psi_{ij}(X_1, \ldots, X_n)=1]\le \exp\Big(-\frac{n}{4}\hels(p_i,p_j)\Big).
\]
\end{lemma}
We also need the following result.

\begin{lemma}[Lemma 32.11 of \cite{polyanskiy2025information}]\label{lemma:local-net}
        For any $p \in \mathcal{P}$, $\eta \ge \eps$, and an integer $k \ge 0$,
        \begin{equation*}
            N_{\hel}(B_{\hel}(p,2^k \eta) \cap \mathcal{P}, \eta/2) \le \nloc(\mathcal{P},\eps)^{k+1}.
        \end{equation*}
\end{lemma}
\paragraph{Proof of \Cref{lem:lecam-birge-local}.}
Set $\eps=\eps_n$ and $R=4\eps$. Let $\{p_1,\ldots,p_N\}$ be a maximal $\eps$-packing. There is an index, which we rename as $1$, such that $\hel(p^\star,p_1)\le \eps$. We run the tests, compute $T_i$, and choose $\widehat{p}=p_{i^\star}$ as above. By the triangle inequality and the definition of $i^\star$,
\[
\hel(\widehat{p},p^\star)\le \hel(\widehat{p},p_1)+\hel(p_1,p^\star)\le \max\{R,T_1\}+\eps\le 5\eps+T_1.
\]
Hence, for any $t\ge 4$,
\[
\P\big(\hel(\widehat{p},p^\star)>5\eps+4t\eps\big)\le \P(T_1\ge tR).
\]

Define the Hellinger shells $A_k=\{p\in\mathcal{P}: 2^kR\le \hel(p_1,p)<2^{k+1}R\}$ and set $G_k=\{p_1,\ldots,p_N\}\cap A_k$. By a union bound over $G_k$ and Lemma~\ref{lemma:pairwise-composite},
\[
\P(T_1\in[2^kR,2^{k+1}R))
\le \sum_{j:\ p_j\in G_k}\exp\Big(-\frac{n}{4}\hels(p_1,p_j)\Big)
\le |G_k|\exp\Big(-\frac{n}{4}(2^kR)^2\Big).
\]
Since $\hel(p_1,p^\star)\le \eps$, we have
\[
B_{\hel}(p_1,2^{k+1}R)\subset B_{\hel}(p^\star,2^{k+1}R+\eps)\subset B_{\hel}(p^\star,2^{k+4}\eps),
\]
and by Lemma~\ref{lemma:local-net},
\[
|G_k|\le N_{\hel}\big(B_{\hel}(p_1,2^{k+1}R)\cap\mathcal{P},\ \eps/2\big)\le \nloc(\mathcal{P},\eps)^{k+5}.
\]
Using the fixed point,
\[
\P(T_1\ge tR)\le \sum_{k\ge \lfloor\log_2 t\rfloor}\exp\big(n\eps^2(k+5)-4n\eps^2\,4^k\big)\le \sum_{k\ge \lfloor\log_2 t\rfloor}\exp\big(-n\eps^2\,4^k\big).
\]
Since $2^{\lfloor\log_2 t\rfloor}\le t<2^{\lfloor\log_2 t\rfloor+1}$ implies $4^{\lfloor\log_2 t\rfloor}\ge t^2/4$, we have
$
\P(T_1\ge tR)\le 2\,\exp\Big(-\frac{n\eps^2\,t^2}{4}\Big).
$
We choose
$
t=\max\Big\{4,\sqrt{ \frac{4}{n\eps^2}\log\frac{2}{\delta}}\ \Big\},
$
which implies $2\exp(-n\eps^2 t^2/4)\le \delta$. Therefore, with probability at least $1-\delta$,
\[
\hel(\widehat{p},p^\star)\le 5\eps+4t\eps\le 21\eps+8\sqrt{\frac{\log(2/\delta)}{n}}.
\]
This gives the claim. Finally, for any $s\ge 0$, choosing $\delta(s)=\min\left\{1, 2 \exp(-n s^{2}/64)\right\}$ gives
\[
\P\left(\hel(p^\star,\hat p) > 21\eps_n + s\right)\le \min\left\{1, 2 \exp(-n s^{2}/64)\right\}.
\]
Let $s_0=8\sqrt{\log 2 / n}$. Integrating the tail, we have
\[
\E\hel(p^\star,\widehat{p})
\le 21\eps_n + \int_{0}^{s_0} 1\,ds + \int_{s_0}^{\infty} 2 \exp(-n s^{2}/64)ds
\le 21\eps_n + \frac{11}{\sqrt{n}}.
\]
The second claim follows. \qed

\section{Other deferred proofs}
\label{sec:remainingproofs}
\subsection[Proof of Lemma~\ref{lem:kltohellinger}]{Proof of \Cref{lem:kltohellinger}}

Let $h(x)={p(x)}/{q(x)}$. By the standard representation we have
\[
\kl(p,q)=\int p\log\frac{p}{q}\,d\mu
=\int q\bigl(h\log h-h+1\bigr)\,d\mu.
\]
We invoke the elementary pointwise inequality, valid for all $h>0$,
\begin{equation}\label{eq:calc-ineq}
h\log h-h+1 \le \frac{1}{(\sqrt{e}-1)^2}(\sqrt{h}-1)^2 \max\{1,\log h\}.
\end{equation}
The inequality \eqref{eq:calc-ineq} can be shown by treating $h\le e$ and $h> e$ separately and checking that
$\frac{h\log h-h+1}{(\sqrt{h}-1)^2}$ is increasing on $(0,e]$ while
$\frac{h\log h-h+1}{(\sqrt{h}-1)^2\log h}$ is decreasing on $[e,\infty)$; both meet at value $(\sqrt e-1)^{-2}$ when $h=e$. Applying \eqref{eq:calc-ineq} with $h=p/q$ and $m=\max\left\{1,\ \sup_{x\in\X}\log\left(\frac{p(x)}{q(x)}\right)\right\}$,
we obtain
\[
\kl(p,q) \le \frac{m}{(\sqrt{e}-1)^2}\int q(\sqrt{h}-1)^2d\mu.
\]
Finally, observe that $\int q(\sqrt{h}-1)^2\,d\mu=\int(\sqrt p-\sqrt q)^2d\mu=2\hels(p,q)$. Substituting yields the claim. \qed

\subsection[Proof of Corollary~\ref{cor:eps-pareto}]{Proof of \Cref{cor:eps-pareto}}

We start with a general amplification lemma that converts a constant-factor ratio cover into a $(1+\eps)$-ratio cover.

\begin{claim}\label{claim:amplify-ratio}
For $0<\eps\le 1$ and $C \ge 1$, let $B \subset K$ be a $C$-ratio cover of a convex and compact set $K \subset \R_+^d$.
Then there exists a subset $A \subset K$ with
\[
|A|\le|B|\left(3 + \frac{4\log_2(4C/\eps)}{\eps}\right)^{d-1}
\]
that is a $(1+\eps)$-ratio cover of $K$.
\end{claim}

\begin{proof}
Fix $b\in B$. We will add to $A$ several points of the form
\[
z  = (\eps/4)b + (1-\eps/4)w,
\]
where $w\in K$ is chosen so that its first $d-1$ coordinates fall into prescribed multiplicative bins relative to $b$.

We now describe this discretization. Fix $\theta\in K$, and choose $b\in B$ that $C$-ratio covers $\theta$ (if there are several, choose one arbitrarily).
Then for each coordinate $i\in[d]$ we have $\theta_i \le C b_i$, and hence
\begin{equation}
\label{eq:b-bound}
\theta_i \le (\eps/4)b_i \cdot \frac{4C}{\eps}.
\end{equation}
We split the possible values of $\theta_i$ into the following cases relative to $b_i$:
\begin{itemize}[leftmargin=2em]
\item \emph{Case 1:} $\theta_i \le (\eps/4)\,b_i$;
\item \emph{Case 2:} $\theta_i > (\eps/4)\,b_i$, in which case \eqref{eq:b-bound} implies that
\[
\theta_i \in \Bigl[(\eps/4)\,b_i\,(1+\eps/4)^j,\ (\eps/4)\,b_i\,(1+\eps/4)^{j+1}\Bigr]
\quad\text{for some } j \in \Bigl\{0,\dots,\Bigl\lceil \log_{1+\eps/4}\Bigl(\frac{4C}{\eps}\Bigr)\Bigr\rceil\Bigr\}.
\]
\end{itemize}

We now define the set $A$.
For each $b\in B$, and for each possible choice of cases/bins for the first $d-1$ coordinates, we include one point
\[
z = (\eps/4)\,b + (1-\eps/4)\,w,
\]
where $w\in K$ has the specified cases/bins on coordinates $1,\dots,d-1$ (if such a $w$ exists).
If such a $w$ exists, we choose $w$ to maximize the $d$-th coordinate among all points in $K$ with this bin pattern.
This maximizer exists since $K$ is compact and the bin constraints above are closed.
There are at most $\bigl(2+\lceil \log_{1+\eps/4}(4C/\eps)\rceil\bigr)^{d-1}$ bin patterns for the first $d-1$ coordinates, so
\begin{align*}
|A|
&\le |B|\Bigl(2+\Bigl\lceil \log_{1+\eps/4}\Bigl(\frac{4C}{\eps}\Bigr)\Bigr\rceil\Bigr)^{d-1}
\le |B|\Bigl(3+\frac{\log_2(4C/\eps)}{\log_2(1+\eps/4)}\Bigr)^{d-1}\\
&\le |B|\Bigl(3+\frac{4\log_2(4C/\eps)}{\eps}\Bigr)^{d-1},
\end{align*}
where in the last step we used $\log_2(1+x)\ge x$ for $0\le x\le 1$ and $\eps/4\le 1/4$.

It remains to show that $A$ is a $(1+\eps)$-ratio cover of $K$.
Fix $\theta\in K$, and let $b\in B$ be a $C$-ratio cover of $\theta$.
By construction of $A$, there exists $z=(\eps/4)\,b+(1-\eps/4)\,w\in A$ such that $w$ has the same bin pattern as $\theta$ on the first $d-1$ coordinates, and in addition $\theta_d \le w_d$ (since $w$ maximizes the $d$-th coordinate within that pattern).
We check coordinatewise domination.

For $i\in[d-1]$, if $\theta_i$ falls in Case~1 then $\theta_i \le (\eps/4)\,b_i \le z_i$.
Otherwise $\theta_i$ is in Case~2, and $\theta_i$ and $w_i$ lie in the same interval $[L,U]$ with $U=(1+\eps/4)L$ and $w_i\ge L$, so $\theta_i \le U \le (1+\eps/4)\,w_i$.
Using $z_i \ge (1-\eps/4)\,w_i$,
\[
\theta_i \le (1+\eps/4)\,w_i \le (1+\eps/4)\,\frac{z_i}{1-\eps/4}
\le (1+\eps/4)(1+\eps/2)\,z_i \le (1+\eps)\,z_i,
\]
where we used $1/(1-x)\le 1+2x$ for $x\in[0,1/2]$ (here $x=\eps/4$) and the fact that $(1+\eps/4)(1+\eps/2)\le 1+\eps$ for $0<\eps\le 1$.
Finally, for the last coordinate, $\theta_d \le w_d \le z_d/(1-\eps/4)\le (1+\eps)\,z_d$.
Thus $\theta_i \le (1+\eps)z_i$ holds for all $i\in[d]$, proving that $A$ is a $(1+\eps)$-ratio cover.
\end{proof}

To conclude the theorem, apply \Cref{claim:amplify-ratio} with $K$ as given, $C=32$, and $B$ a $32$-ratio cover of $K$ of size at most $2^{8d}$ provided by \Cref{thm:ratio-cover}.
This yields an $\eps$-approximate Pareto curve of size
\[
|A|
 \le 2^{8d}\,\Bigl(3 + \frac{4\log_2(4\cdot 32/\eps)}{\eps}\Bigr)^{d-1}
 = 2^{8d}\,\Bigl(3 + \frac{4\log_2(128/\eps)}{\eps}\Bigr)^{d-1}.
\]
The claim follows. \qed

\subsection[Proof of Lemma~\ref{lemma:kl-from-ball}]{Proof of \Cref{lemma:kl-from-ball}}

Fix $q\in\mathcal P$ and set $\Delta\triangleq 1\vee \hels(p,q)/\eps$ and $\lambda\triangleq 1/\Delta$.
Define $q_\lambda\triangleq (1-\lambda)p+\lambda q\in\mathcal P$.
By convexity of $\hels(p,\cdot)$ we have
\[
\hels(p,q_\lambda)=\hels\bigl(p,(1-\lambda)p+\lambda q\bigr)\le \lambda\,\hels(p,q)\le \eps,
\]
so $q_\lambda\in\mathcal Q$. Pick $\tilde q\in S$ such that $q_\lambda\le \alpha\,\tilde q$ pointwise.
Since $q_\lambda\ge \lambda q$, it follows that $q\le \alpha\Delta\,\tilde q$ and hence
\[
\sup_{x\in\X}\log\Bigl(\frac{q(x)}{\tilde q(x)}\Bigr)\le \log(\alpha\Delta).
\]
Also $\tilde q\in\mathcal Q$, so $\hels(p,\tilde q)\le \eps$, and by the triangle inequality, it holds that 
\[
\hels(q,\tilde q)=\hel(q,\tilde q)^2
\le \bigl(\hel(q,p)+\hel(p,\tilde q)\bigr)^2
\le 2\hels(p,q)+2\hels(p,\tilde q)
\le 2\bigl(\hels(p,q)+\eps\bigr).
\]
Applying Lemma~\ref{lem:kltohellinger} gives
\begin{align*}
\kl(q,\tilde q)
&\le
\frac{2}{(\sqrt e-1)^2}\hels(q,\tilde q)\,
\max\left\{1,\sup_{x\in\X}\log\Bigl(\frac{q(x)}{\tilde q(x)}\Bigr)\right\}
\\
&\le
\frac{4}{(\sqrt e-1)^2}(\hels(p,q)+\eps)\max\{1,\log(\alpha\Delta)\}.
\end{align*}
Since $\alpha\ge 2$ and $\Delta\ge 1$, we have $\log(\alpha\Delta)\ge \log 2$, so $\max\{1,\log(\alpha\Delta)\}\le (\log(\alpha\Delta))/\log 2$.
Using $\frac{4}{(\sqrt e-1)^2\log 2}<30$, we obtain
\[
\kl(q,\tilde q)\le 30(\hels(p,q)+\eps)\log(\alpha\Delta).
\]
Finally, $\log(\alpha\Delta)=\max\left\{\log(\alpha),\ \log\left(\alpha\hels(p,q)/\eps\right)\right\}$ by definition of $\Delta$.
Taking the minimum over $\tilde q\in S$ yields the claim.\qed

\subsection[Proof of Proposition~\ref{prop:misspecmodelaggregation}]{Proof of \Cref{prop:misspecmodelaggregation}}
First, we present the required estimator, which is simpler than the one we used in \Cref{thm:modelaggregation}.
\begin{framed}
\begin{itemize}
\item Construct a class of densities 
\begin{equation}
\label{eq:newwprimeset}
\mathcal{Q}
=
\left\{
\left(1 - \frac{1}{n}\right)p
+ \frac{1}{nM}\sum_{q \in \mathcal{P}} q
:\ p \in \mathcal{P}
\right\}.
\end{equation}
\item Run the density estimator of \Cref{lem:onlinetobatch} over $\mathcal{Q}$ to output the final density $\widehat{p}$.
\end{itemize}
\end{framed}

Let $p^\star_{\mathcal{P}}$ be any minimizer of $\kl(p^\star, p)$ over $p \in \mathcal{P}$ and let
\[
\widetilde{p}
=
\left(1 - \frac{1}{n}\right)p^\star_{\mathcal{P}}
+ \frac{1}{nM}\sum_{q \in \mathcal{P}} q.
\]
Observe that $\widetilde{p} \in \Qs$ and that for each $p, q \in \Qs$ we have
\[
\sup_{p,q \in \Qs}\sup_{x\in\X}\log\left(\frac{p(x)}{q(x)}\right)
\le \log\bigl(1 + nM\bigr)
\le \log(2nM).
\]
Therefore, by \Cref{lem:onlinetobatch} we have, with probability at least $1 - \delta$,
\[
-\E\log(\widehat{p}(X)) + \E\log(\widetilde{p}(X))
\le
\frac{2\log(M) + \frac{25}{4}\left(e-2\right)\left(\log(2nM)+1\right)\log\left(\frac{1}{\delta}\right)}{n}.
\]
Finally, we show that
\[
\E\log\bigl(p^\star_{\mathcal{P}}(X)\bigr)
\le
\E\log\bigl(\widetilde{p}(X)\bigr)
+ \frac{2}{n},
\]
which, by combining this with the previous inequality and simplifying the constants, immediately proves the claim. Indeed, by the definition of $\widetilde{p}$ we have, for all $x\in\X$,
\begin{equation}
\label{eq:logratioofdens}
\log\left(\frac{p^\star_{\mathcal{P}}(x)}{\widetilde{p}(x)}\right)
\le \log\left(\frac{1}{1 - 1/n}\right)
\le \frac{2}{n},
\end{equation}
whenever $n \ge 2$. When $n = 1$, we have $\mathcal{Q} = \big\{\frac{1}{\M}\sum\limits_{q \in \mathcal P}q\big\}$ and the claim follows since $\widehat{p} = \frac{1}{\M}\sum\limits_{q \in \mathcal P}q$.\qed

\subsection[Proof of Proposition~\ref{prop:misspecconvexaggregation}]{Proof of \Cref{prop:misspecconvexaggregation}}

We assume throughout that $n\ge 2$. When $n=1$, $\Qs$ contains a single element, namely
$
\Qs=\left\{\frac{1}{\M}\sum_{q\in\mathcal P}q\right\},
$
so $\widehat p$ is this deterministic density and the stated bound is immediate.

The proof exploits the curvature of the logarithmic loss on the ratio-bounded convex class $\Qs$
via an offset-process argument (see \cite{liang2015learning,mehta2017expconcave_statistical,vijaykumar2021localization} for related derivations).
Recall that $\widehat p$ is the maximum-likelihood estimator over $\Qs$, namely
\[
\widehat{p} \in \arg\min_{p \in \Qs}\wh L(p),
\qquad
\wh L(p)\triangleq -\frac{1}{n}\sum_{i=1}^n \log p(X_i),
\qquad
L(p)\triangleq -\E\log p(X),
\]
where $X\sim p^\star$ and $\E$ denotes expectation under $p^\star$.
As in the proof of \Cref{prop:misspecmodelaggregation}, the smoothing in the definition of $\Qs$ implies that
\begin{equation}\label{eq:ratioofdens-again}
\sup_{p,q \in \Qs}\sup_{x\in\X}\Bigl|\log\Bigl(\frac{p(x)}{q(x)}\Bigr)\Bigr|
\le \log(2nM).
\end{equation}
We start with the following symmetrization lemma.

\begin{lemma} \label{lem:mgf-start-misspec}
Let $p^\star_{\Qs}\in\arg\min_{p\in\Qs} L(p)$.
For every $\lambda>0$ it holds that
\begin{equation}\label{eq:mgf-start}
\E\exp\Big(\lambda\big(L(\widehat p)-L(p^\star_{\Qs})\big)\Big)
\le \E\exp\Big(\lambda\sup_{p\in\Qs} Z(p)\Big),
\end{equation}
where
\[
Z(p)\triangleq \frac{1}{n}\sum_{i=1}^n
\left[
  4\varepsilon_i\log\left(\frac{p^\star_{\Qs}(X_i)}{p(X_i)}\right)
  -\frac{1}{2(\log(2nM)+1)}\log^2\left(\frac{p^\star_{\Qs}(X_i)}{p(X_i)}\right)
\right],
\]
and $(\varepsilon_i)_{i=1}^n$ are i.i.d.\ Rademacher signs independent of $(X_i)_{i=1}^n$.
\end{lemma}

\begin{proof}
Define the log-density ratio
$
r_p(x)\triangleq \log\left(\frac{p^\star_{\Qs}(x)}{p(x)}\right).
$
Let $(X_1',\dots,X_n')$ be an independent ghost sample (i.i.d. as $(X_i)$) and set
\[
P_n r_p = \frac1n\sum_{i=1}^n r_p(X_i),
\quad
P_n(r_p^2) = \frac1n\sum_{i=1}^n r_p(X_i)^2,
\quad
P_n' r_p = \frac1n\sum_{i=1}^n r_p(X_i'),
\quad
P_n'(r_p^2) = \frac1n\sum_{i=1}^n r_p(X_i')^2.
\]
By convexity of $\Qs$, the midpoint
$
\bar p = \frac{\widehat p+p^\star_{\Qs}}{2}
$
belongs to $\Qs$. By the definition of $\widehat p$, it holds that
\begin{equation}\label{eq:erm-midpoint}
\wh L(\widehat p)\le \wh L(\bar p).
\end{equation}
Using \eqref{eq:ratioofdens-again} we can apply \Cref{lem:logloss-Tstyle}, which gives for all $i \in [n]$,
\[
-\log(\bar{p}(X_i)) = -\log\Big(\tfrac{\widehat p(X_i)+p^\star_{\Qs}(X_i)}{2}\Big)
\le \frac{-\log \widehat p(X_i)-\log p^\star_{\Qs}(X_i)}{2}
-\frac{1}{8(\log(2nM)+1)}\,r_{\widehat p}(X_i)^2.
\]
Averaging over $i\in [n]$, combining with \eqref{eq:erm-midpoint} and rearranging we obtain
\begin{equation}\label{eq:self-bounding}
P_n r_{\widehat p}\le -\frac{1}{4(\log(2nM)+1)}P_n\big(r_{\widehat p}^2\big).
\end{equation}

We repeat the same lines at the population level.
Fix any $p\in\Qs$ and let $\tilde p = (p+p^\star_{\Qs})/2\in\Qs$.
Because $p^\star_{\Qs}$ is a population minimizer over $\Qs$, we have
\begin{equation}\label{eq:pop-min}
L(p^\star_{\Qs}) \le L(\tilde p).
\end{equation}
Applying Lemma~\ref{lem:logloss-Tstyle} as above, using \eqref{eq:pop-min}, taking the expectation and rearranging we obtain that for all $p \in \Qs$,
\begin{equation}\label{eq:bernstein-pop}
\E[r_p(X)] \ge \frac{1}{4(\log(2nM)+1)}\E[r_p(X)^2].
\end{equation}
We have $L(\widehat p)-L(p^\star_{\Qs})
= \E r_{\widehat p}(X)$. For brevity of notation, set
$
\alpha = \frac{1}{4(\log(2nM)+1)}.
$
The following lines of inequalities hold (conditioned on $\widehat{p}$):
\begin{align*}
\E\, r_{\widehat p}(X)
&\stackrel{\eqref{eq:bernstein-pop}}{\le}
2\Big(\E\, r_{\widehat p}(X)-\frac{\alpha}{2}\E\big[r_{\widehat p}(X)^2\big]\Big)\\
&=
2\Big(\big(\E-P_n\big) r_{\widehat p}+P_n r_{\widehat p}-\frac{\alpha}{2}\E\big[r_{\widehat p}(X)^2\big]\Big)\\
&\stackrel{\eqref{eq:self-bounding}}{\le}
2\Big(\big(\E-P_n\big) r_{\widehat p}-\alpha\,P_n\big(r_{\widehat p}^2\big)-\frac{\alpha}{2}\E\big[r_{\widehat p}(X)^2\big]\Big)\\
&\le
\sup_{p\in\Qs}\Big\{2\big(\E-P_n\big) r_{p}-2\alpha\,P_n\big(r_{p}^2\big)-\alpha\,\E\big[r_{p}(X)^2\big]\Big\}\\
&\le
\sup_{p\in\Qs}\Big\{2\big(\E-P_n\big) r_{p}-\alpha P_n\big(r_{p}^2\big)-\alpha\E\big[r_{p}(X)^2\big]\Big\}.
\end{align*}
Finally, combining the above inequality with the standard symmetrization argument (that is, when appropriate $\E$ corresponds to taking the expectation with respect to the joint distribution of $(X_i)$, $(X_i^\prime)$ and $(\varepsilon_i)$) we have for every $\lambda>0$,
\begin{align*}
\E\exp\Big(\lambda\big(L(\widehat p)-L(p^\star_{\Qs})\big)\Big)
&\le \E\exp\big(\lambda\sup_{p\in\Qs}\Big\{2\big(\E-P_n\big) r_{p}-\alpha\,P_n\big(r_{p}^2\big)-\alpha\,\E\big[r_{p}(X)^2\big]\Big\}\big)
\\&\le \E\exp\big(\lambda\sup_{p\in\Qs}\Big\{2\big(P_n'-P_n\big) r_{p}-{\alpha}P_n\big(r_{p}^2\big)-{\alpha}P_n'\big(r_{p}^2\big)\Big\}\big),
\\
&=\E\exp\Big(
\lambda\sup_{p\in\Qs}
\frac1n\sum_{i=1}^n
\Big[
 2\varepsilon_i\big(r_p(X_i')-r_p(X_i)\big)
 -\alpha\big(r_p(X_i)^2+r_p(X_i')^2\big)
\Big]
\Big)
\\
&\le\E\exp\Big(
\lambda\sup_{p\in\Qs}
\frac1n\sum_{i=1}^n
\Big[
 4\varepsilon_i r_p(X_i)
 -2\alpha r_p(X_i)^2
\Big]
\Big),
\end{align*}
where in the last line we split the terms involving $P_n$ and $P_n\prime$ and use that for a pair of random variables $Y, Y'$ having the same distribution, it holds that $\E\exp(Y + Y^\prime) \le \sqrt{\E\exp(2Y)}\sqrt{\E\exp(2Y^\prime)} = \E\exp(2Y)$. The claim follows.
\end{proof}

The next lemma is about the Lipschitzness of the loss class indexed by $\Qs$. For $\alpha \in \Delta_{\M-1}$, we denote
\[
w_\alpha(x)\triangleq \Bigl(1-\frac1n\Bigr)\sum_{j=1}^M \alpha_j p_j(x)+\frac{1}{nM}\sum_{j=1}^M p_j(x),
\]
so that $\Qs=\{w_\alpha:\alpha\in\Delta_{\M-1}\}$, where $\Delta_{\M-1}$ is the simplex in $\R^M$.

\begin{lemma}
    \label{lem:lipschitz}
    For any $\alpha, \beta \in \Delta_{\M-1}$ it holds that
    \[
    \sup\limits_{x \in \X}|\log(w_\alpha(x)) - \log(w_\beta(x))| \le nM\|\alpha - \beta\|_1.
    \]
\end{lemma}
\begin{proof}
Fix $\alpha,\beta\in\Delta_{\M-1}$ and $x\in\X$. It holds that
\[
|w_\alpha(x)-w_\beta(x)| = \left(1 - \frac{1}{n}\right)\left|\sum\limits_{j = 1}^M(\alpha_j - \beta_j)p_j(x)\right|
\le \sum_{j=1}^M |\alpha_j-\beta_j|\,p_j(x)
\le \|\alpha-\beta\|_1 \sum_{j=1}^M p_j(x).
\]
By the mean value theorem we have
\[
|\log(w_\alpha(x))-\log(w_\beta(x))|
\le \frac{|w_\alpha(x)-w_\beta(x)|}{\min\{w_\alpha(x), w_\beta(x)\}}.
\]
Observe that
$
w_\alpha(x) \ge \frac{1}{nM}\sum_{j=1}^M p_j(x)$ and 
$w_\beta(x) \ge \frac{1}{nM}\sum_{j=1}^M p_j(x)$.
Combining the above inequalities yields
\[
|\log(w_\alpha(x))-\log(w_\beta(x))|
\le
\frac{\|\alpha-\beta\|_1 \sum_{j=1}^M p_j(x)}{\frac{1}{nM}\sum_{j=1}^M p_j(x)}
= nM\|\alpha-\beta\|_1.
\]
The claim follows.
\end{proof}

We fix $\rho\in(0,1]$ and let $\mathcal A_\rho$ be the minimal $\rho$-net of $\Delta_{\M-1}$ with respect to the $\|\cdot\|_1$ distance.
Define the finite subclass $\mathcal Q_\rho= \{w_\alpha:\alpha\in\mathcal A_\rho\}\subseteq \mathcal Q$. Let $p$ be any element in $\Qs$ and $p_{\rho}$ be the $\rho$-close element to it in $\mathcal{Q}_{\rho}$ with respect to the net $\mathcal{A}_\rho$.
By \Cref{lem:lipschitz} and \eqref{eq:ratioofdens-again} we have
\[
Z(p) - Z(p_\rho) \le 4nM\rho + \frac{\log(2nM)}{\log(2nM)+1}nM\rho \le 5nM\rho,
\]
which implies for any $\lambda > 0$,
\begin{align*}
\E_{\varepsilon}\exp\Big(\lambda\sup_{p\in\Qs} Z(p)\Big)
&\le
\exp\Big(5\lambda nM\rho\Big)\,
\E_{\varepsilon}\exp\Big(\lambda\max_{q\in\Qs_\rho} Z(q)\Big)
\\
&\le
|\Qs_\rho|\exp\Big(5\lambda nM\rho\Big)\max\limits_{p \in \mathcal{Q}_{\rho}}\E_{\varepsilon}\exp\Big(\lambda Z(p)\Big).
\end{align*}
Using $\E\exp(\lambda \varepsilon_i) \le \exp(\lambda^2/2)$ we have for any $\lambda \in \left[0, \frac{n}{16(\log(2nM)+1)}\right]$ and $p \in \mathcal{Q}$,
\begin{align*}
\E_{\varepsilon}\exp\Big(\lambda Z(p)\Big)&=\E_{\varepsilon}\exp\left(\frac{\lambda}{n}\sum\limits_{i =1}^n\left(4\varepsilon_i\log\left(\frac{p^\star_{\Qs}(X_i)}{p(X_i)}\right)
-\frac{1}{2(\log(2nM)+1)}\log^2\left(\frac{p^\star_{\Qs}(X_i)}{p(X_i)}\right)\right)\right) 
\\
&\le \exp\left(\sum\limits_{i =1}^n\left(\frac{8\lambda^2}{n^2}\log^2\left(\frac{p^\star_{\Qs}(X_i)}{p(X_i)}\right)
-\frac{\lambda}{2n(\log(2nM)+1)}\log^2\left(\frac{p^\star_{\Qs}(X_i)}{p(X_i)}\right)\right)\right) 
\\
&\le 1.
\end{align*}
Taking expectation over $(X_i)$ and using \eqref{eq:mgf-start}, we obtain that for any $\lambda \in \left[0, \frac{n}{16(\log(2nM)+1)}\right]$,
\begin{equation}\label{eq:mgf-excess-final}
\E\exp\Big(\lambda\big(L(\widehat p)-L(p^\star_{\Qs})\big)\Big)
\le
|\Qs_\rho|\exp\Big(5\lambda nM\rho\Big) \le \exp\left(M\log(3/\rho) + 5\lambda nM\rho\right),
\end{equation}
where we used the standard volumetric argument (see \cite[Lemma 5.7]{wainwright2019high}) to bound $|\mathcal{Q}_{\rho}|$ for $\rho \le 1$. This immediately implies that, with probability at least $1-\delta$,
\begin{equation}\label{eq:hp-excess-risk-W}
L(\widehat p)-L(p^\star_{\Qs})
\le
5nM\rho + \frac{1}{\lambda}\Big(M\log(3/\rho)+\log(1/\delta)\Big).
\end{equation}
Taking $\lambda=\frac{n}{16(\log(2nM)+1)}$ and $\rho = \frac{1}{3n^2M} \le 1$, we obtain, using $L(\widehat p)-L(p^\star_{\Qs})
=\kl(p^\star,\widehat p)-\kl(p^\star,p^\star_{\Qs})$,
\begin{equation}\label{eq:hp-excess-risk-W-simplified}
\kl(p^\star,\widehat p)-\kl(p^\star,p^\star_{\Qs})
\le
\frac{5}{3n} + \frac{16(\log(2nM)+1)(M\log(9n^2M)+\log(1/\delta))}{n}.
\end{equation}
Finally, repeating the argument in \eqref{eq:logratioofdens} we have
\begin{equation}\label{eq:infW-vs-infconv}
\kl(p^\star,p^\star_{\Qs})
\le
\inf_{p\in\conv(\mathcal P)}\kl(p^\star,p)+\frac{2}{n}.
\end{equation}
The claim follows by combining \eqref{eq:hp-excess-risk-W-simplified} and \eqref{eq:infW-vs-infconv} and adjusting the constants. \qed
\end{document}